\documentclass[a4paper]{article}
\usepackage{hyperref}
\usepackage{amsmath}
\usepackage{amsrefs}
\usepackage{amssymb}
\usepackage{amsthm}
\usepackage{array}
\usepackage{bbding}
\usepackage{mathpazo}
\usepackage{shuffle}
\usepackage{stmaryrd}
\usepackage{todonotes}
\usepackage{verbatim}

\usepackage{tikz}
\usetikzlibrary{cd, decorations.markings}
\usepackage{braids}

\tikzset{>=stealth}

\tikzset{slashed/.style={
    decoration={markings,
      mark=at position 0.5 with {
        \node[transform shape] (tempnode) {\tiny $\backslash$};
      }
    },
    postaction={decorate}}}
\tikzset{double slashed/.style={
    decoration={markings,
      mark=at position 0.5 with {
        \node[transform shape] (tempnode) {\tiny $\backslash\backslash$};
      }
    },
    postaction={decorate}}}

\definecolor{cbblue}{RGB}{100,134,255}
\definecolor{cbred}{RGB}{220,37,127}
\definecolor{cborange}{RGB}{225,176,0}
\definecolor{cbteal}{RGB}{110,180,130}

\newcommand{\1}{\mathbf{1}}
\newcommand{\Ab}{\operatorname{Ab}}
\newcommand{\Alg}{\operatorname{Alg}}

\newcommand{\Cat}{\operatorname{Cat}}

\newcommand{\fin}[1]{#1^\textrm{fin}}
\newcommand{\Hom}{\operatorname{Hom}}
\newcommand{\inj}{\mathrm{inj}}
\newcommand{\isom}{\cong}
\newcommand{\I}{\mathbb{1}}

\newcommand{\Lax}{\operatorname{Lax}}
\newcommand{\Lin}{\operatorname{Lin}}

\newcommand{\Mod}{\operatorname{Mod}}

\newcommand{\mono}{\mathrm{mono}}
\newcommand{\Ob}{\operatorname{Ob}}
\newcommand{\op}{\mathrm{op}}
\newcommand{\Poset}{\operatorname{Poset}}
\newcommand{\Rel}{\operatorname{Rel}}
\newcommand{\Set}{\operatorname{Set}}
\newcommand{\Vect}{\operatorname{Vect}}
\newcommand{\bB}{\mathbb{B}}
\newcommand{\bD}{\mathbb{D}}
\newcommand{\bN}{\mathbb{N}}
\newcommand{\bP}{\mathbb{P}}

\newcommand{\bZ}{\mathbb{Z}}

\newcommand{\cB}{\mathcal{B}}
\newcommand{\cC}{\mathcal{C}}
\newcommand{\cD}{\mathcal{D}}
\newcommand{\cE}{\mathcal{E}}

\newcommand{\cP}{\mathcal{P}}
\newcommand{\bulletblue}{\textcolor{cbblue}{\bullet}}
\newcommand{\bulletred}{\textcolor{cbred}{\bullet}}
\newcommand{\bulletorange}{\textcolor{cborange}{\bullet}}
\newcommand{\bulletteal}{\textcolor{cbteal}{\bullet}}

\newtheorem{thm}{Theorem}[subsection]
\newtheorem{lemma}[thm]{Lemma}
\newtheorem{prop}[thm]{Proposition}
\newtheorem{conj}[thm]{Conjecture}
\theoremstyle{definition}
\newtheorem{defn}[thm]{Definition}
\theoremstyle{remark}
\newtheorem{remark}[thm]{Remark}
\newtheorem{example}[thm]{Example}

\newcolumntype{C}{>{$}c<{$}}

\newcounter{grumbles}

\newcommand{\countgrumbles}
  {\ifnum\value{grumbles}=0
      {}
   \else\ifnum\value{grumbles}=1
      \message{^^J^^J*** There is 1 grumble in total. ***^^J^^J}
   \else
      \message{^^J^^J*** There are \arabic{grumbles} grumbles in total. ***^^J^^J}
   \fi\fi}

\title{Interacting Monoidal Structures with Applications in Computing}
\author{James Cranch \and Georg Struth}

\begin{document}
\maketitle

\begin{abstract}
  With a view on applications in computing, in particular concurrency
  theory and higher-dimensional rewriting, we develop notions of
  $n$-fold monoid and comonoid objects in $n$-fold monoidal categories
  and bicategories. We present a series of examples for these
  structures from various domains, including a categorical model for a
  communication protocol and a lax $n$-fold relational monoid, which
  has previously been used implicitly for higher-dimensional rewriting
  and which specialises in a natural way to strict $n$-categories. A
  special set of examples is built around modules and algebras of the
  boolean semiring, which allows us to deal with semilattices,
  additively idempotent semirings and quantales using tools from
  classical algebra.

  \vspace{\baselineskip}


  \vspace{\baselineskip}

  \noindent \textbf{MSC subject classifiers}: 18D10, 68Q85, 06A12.

\end{abstract}

\tableofcontents


\section{Introduction}
\label{s:introduction}

Monoidal structure appears in many contexts and guises in
computing. The sequential and parallel compositions of imperative
programs give obviously rise to monoids; the effects of imperative
programs on program stores the can be seen as monoid actions. In
automata theory, monoids can be associated in a canonical way with
formal languages, and types of languages can be characterised in terms
of their properties~\cite{Sakarovitch}. Computational resources are
considered as monoidal structure in Petri nets~\cite{MesMon}, linear
logic~\cite{Girard} or separation logic~\cite{CalOHeYan}, and
computational effects in functional programs can be modelled as
monads~\cite{Moggi}, which are generalised monoids in suitable
categories.

Multiple interacting monoidal structures appear naturally for instance
in concurrency theory. Examples are the concurrent
monoids~\cite{HoaMolStrWeh} and relational interchange
monoids~\cite{CraDohStr1} considered in the context of concurrent
Kleene algebras~\cite{HoaMolStrWeh} Questions about their
categorification lead directly to questions about double monoid
objects in double monoidal categories, $n$-fold monoid objects in
$n$-fold monoidal categories and lax $n$-fold monoids in monoidal
bicategories.

A concrete example of a relational exchange monoid can be found in the
interleaving semantics of concurrent programs where program executions
are modelled as finite lists of events that occur in temporal
order. The serial composition $\odot$ of two such executions then
corresponds to the concatenation of their event lists; their parallel
composition $\otimes$ yields the set of event lists in which the
events of the component lists have been inserted in their original
order in all possible ways. Both compositions are associative (the
type of $\otimes$ requires extending the operation Kleisli-style to
power sets) and they share the list consisting of the empty event as
their unit. The resulting monoidal structures interact via the
interchange law
\begin{equation*}
  (w\otimes x)\odot (y\otimes z) \subseteq (w\odot y) \otimes (x\odot z)
\end{equation*}
for all event lists $w$, $x$, $y$ and $z$, where on the left-hand side
of the inclusion, $\odot$ is again extended to powersets. The parallel
composition can alternatively be modelled as a ternary relation
$R^z_{xy}\Leftrightarrow z \in x\otimes y$, which describes the
multiplication of a relational monoid -- a monoid object in the
monoidal category $\Rel$ with sets as objects and relations as
arrows~\cites{Rosenthal,KenPar}. The serial composition can be coded
along the same lines as $S^z_{xy} \Leftrightarrow z = x\odot y$ to get
a second relational monoid. The associativity laws of $\otimes$ and
$\odot$ then correspond to associativity for $R$ and $S$ in relational
monoids; the interchange law can be coded in similar fashion as a
relational interchange law.  This yields a relational interchange
monoid~\cite{CraDohStr1}.

As an example of a concurrent monoid rooted in true concurrency,
consider concurrent program executions where events are partially
ordered with respect to (temporal or causal) precedence, so that
antichains indicate concurrent events. A convenient model for such
partial-order concurrency~\cite{Vogler} is formed by isomorphism
classes of event-labelled posets, which are known as partial
words~\cite{Grabowski} or pomsets~\cite{Pratt}. The parallel
composition $\otimes$ of two concurrent program executions is then
modelled by the disjoint union of the underlying pomsets, while their
serial composition $\otimes$ makes every element in the first pomset
precede every element in the second (while preserving the internal
order of the two pomsets). Once again both operations are associative;
they have the empty pomset as their unit.  Further, a program subsumes
a second one if there exists an order preserving map from the first
pomset to the second which is bijective on events. Subsumption induces
a partial order $\preceq$ on pomsets and once again it can be shown
that the two monoidal structures interact via interchange:
\begin{equation*}
   (w\otimes x)\odot (y\otimes z) \preceq (w\odot y) \otimes (x\odot z)
 \end{equation*}
 for all pomsets $w$, $x$, $y$ and $z$. This yields a concurrent
 monoid~\cite{HoaMolStrWeh}. When modelled as a relational monoid, the
 same structure as in the previous example appears~\cite{CraDohStr1}.

 This raises the question about a more general structural explanation
 of these examples, and beyond, in terms of interacting monoid objects
 in suitable monoidal categories using suitable lax monoidal functors,
 as explained above. A systematic exploration of multiple strict
 monoidal structures has been given by Balteanu, Fiedorowcz, Schwänzl
 and Vogt~\cite{BFSV}. Yet the interchange laws in our examples
 require less strictness, which can be found in more recent work by
 Aguiar and Mahajan~\cite{AguMah}. Applications of monoidal categories
 in concurrency theory were pioneered half a century ago by
 Winkowski~\cite{Wink}; Meseguer and Montanari's influential article
 on Petri nets as monoids~\cite{MesMon} is another early
 example. Arguably the earliest use of monoidal categories in
 computing is Hotz's work on electronic circuits~\cite{Hot1}, where an
 early instance of string diagrams appears as well.

 Our main contribution, in response to this question, lies in the
 development of a categorical framework for interacting monoids and
 the discussion of a series of examples, some abstract, some
 concrete.

 For this, we adapt Aguiar and Mahajan's approach to $n$-fold monoidal
 categories~\cite{AguMah}, in which suitable notions of $n$-fold
 monoidal functors, $n$-fold monoid objects and $n$-fold comonoid
 objects can be developed. Beyond that we introduce a notion of
 $n$-fold monoid object in a monoidal bicategory, which we need for
 examples such as the above.

 With this set-up, in the first part of the paper, we discuss a
 variety of examples relevant to concurrency in particular, to
 computing in general and of more structural nature, including the
 concurrent monoids, as instances of $2$-fold monoid objects, and
 relational interchange monoids, as instances of lax $2$-fold monoid
 objects in $\Rel$, with their concrete models discussed above. We
 also present two more substantial examples. First, we show how
 communication protocols can be modelled using the $n$-fold monoidal
 structure available in categories with products and
 coproducts. Second, we explain how strict $n$-categories can be
 formalised as instances of lax $n$-fold relational monoids, and
 further, how the more general $n$-catoids, which have previously been
 proposed as a convenient framework for higher-dimensional
 rewriting~\cite{CMPS}, are equivalent to lax $n$-fold relational
 monoids, thus giving a structural explanation for these higher
 catoids.

 In a second series of examples and applications, we consider algebras
 and coalgebras over the boolean semiring, extending and adapting work
 by Connes and Consani~\cite{ConCon}. This allows us to treat
 semilattices as modules and additively idempotent semirings or
 quantales as monoid objects in modules on the boolean
 semiring. Relational monoids and $n$-fold lax relational monoid
 objects, in particular, arise in an alternative way in this setting,
 and we can extend the relational interchange monoid structure of
 words with respect to concatenation and shuffle in a structured way
 to concurrent monoid structure at language (or powerset) level. On
 the co-side, we consider incidence coalgebras in the sense of Joni
 and Rota~\cite{JonRot}, and show how convolution algebras arise in
 the presence of a comonoid and a monoid object, with special
 consideration of the boolean semiring, where languages with respect
 to word concatenation and shuffle arise once again. Finally we show a
 grading construction similar, but structurally different to
 convolution yields another $n$-fold monoidal structure.

 The overall structure of this article is as follows. In
 Section~\ref{s:monoidal-categories} we recall the basic machinery of
 monoidal categories, while in Section~\ref{s:examples-monoidal} we
 present examples of these structures that are of interest later on,
 including relational monoids. In Section~\ref{s:n-fold-monoidal} we
 first introduce notions of categories with two monoidal
 structures. We then extend this development to $n$-monoidal
 structures and finally define lax $n$-fold monoidal objects in
 bicategories. In Section~\ref{s:examples-nfold} we provide a series
 of examples: from partial orders, category theory itself, concurrency
 theory, categories of graphs and related objects, categorical models
 for communication protocols and finally higher category theory and
 higher-dimensional rewriting. Last but not least, in Section
 \ref{s:boolsemiring}, we specialise this material to the
 consideration of semilattices, additively idempotent semirings and
 similar structures and we extend our considerations from monoid
 objects to comonoids, bimonids, convolution algebras and graded
 objects with $n$-fold monoidal structure.

We suspect that much of this material will be unfamiliar to some
readers. As an attempt to engage with this, even in cases where
abstract definitions are possible, we translate them into explicit
commutative diagrams. This has increased the length, but in our
experience it is easy not to pay too much attention to diagrams.

\section{Monoidal categories}
\label{s:monoidal-categories}

We assume basic knowledge on monoidal categories (see \cite{MacL} for
details, or \cite{AguMah} for a technical account relevant to our
needs).

\subsection{Basic definitions}

We typically write $\I\in\Ob\cC$ for the unit object of a
monoidal category $\cC$, $\otimes:\cC\times\cC\rightarrow\cC$ for its
tensor, $\alpha_{X,Y,Z} : (X\otimes Y)\otimes Z\rightarrow
X\otimes(Y\otimes Z)$ for its associator and $\lambda_X : \I\otimes
X\rightarrow X$ and $\rho_X : X\otimes \I\rightarrow X$ for its
unitors. These natural transformations are subject to coherence
conditions.

We refer to the data $(\otimes, \I, \alpha, \lambda, \rho)$ as a
\emph{monoidal structure} on a category $\cC$.  A monoidal category is
\emph{strict} if $\alpha$, $\lambda$ and $\rho$ are identity natural
transformations.

We abuse language and notation in various ways. We frequently do not
mention $\alpha$, $\lambda$ and $\rho$ explicitly. If several monoidal
structures are present (with the same underlying category or not), we
use sub- and superscripts to distinguish them ad hoc. We may also be
vague about distinguishing a monoidal category from its underlying
category.

\subsection{Braided and symmetric monoid categories}

We further recall two different notions of commutativity for monoidal
categories\cite{MacL} -- we will find use for both.

A \emph{braided monoidal category} is a monoidal category equipped
with a natural isomorphism $\tau_{X,Y}:X\otimes Y\rightarrow
Y\otimes X$ such that the following diagrams commute:
\begin{center}
  \begin{tikzcd}
    (X\otimes Y)\otimes Z\ar[r, "\alpha"]\ar[d, "\tau\otimes1"']&
    X\otimes(Y\otimes Z)\ar[d, "\tau"]\\
    (Y\otimes X)\otimes Z\ar[d, "\alpha"']&
    (Y\otimes Z)\otimes X\ar[d, "\alpha"]\\
    Y\otimes (X\otimes Z)\ar[r, "1\otimes\tau"']&
    Y\otimes (Z\otimes X),
  \end{tikzcd}
  \hfil
  \begin{tikzcd}
    X\otimes(Y\otimes Z)\ar[r, "\alpha^{-1}"]\ar[d, "1\otimes\tau"']&
    (X\otimes Y)\otimes Z\ar[d, "\tau"]\\
    X\otimes(Z\otimes Y)\ar[d, "\alpha^{-1}"']&
    Z\otimes(X\otimes Y)\ar[d, "\alpha^{-1}"]\\
    (X\otimes Z)\otimes Y\ar[r, "\tau\otimes 1"']&
    (Z\otimes X)\otimes Y.
  \end{tikzcd}
\end{center}

A \emph{symmetric monoidal category} is a braided monoidal category
where the natural transformation $\tau_{Y,X}\tau_{X,Y}$ equals the
identity $1_{X\otimes Y}$.

\subsection{Monoidal functors}

There are several notions of functor between monoidal categories.
\begin{defn}
  Let $\cC,\cD$ be monoidal categories. A \emph{lax monoidal functor}
  $F$ from $\cC$ to $\cD$ is a functor $\cC\rightarrow\cD$ between the
  underlying categories together with a morphism
  $\varepsilon:\I_\cD\rightarrow F(\I_\cC)$ (the \emph{unit}) and a
  natural transformation
  $\mu_{X,Y}:FX\otimes_\cD FY\rightarrow F(X\otimes_\cC Y)$ (the
  \emph{multiplication}), such that the following diagrams commute:
  \begin{center}
    \begin{tikzcd}
      (FX\otimes_\cD FY)\otimes_\cD FZ \ar[r,"\alpha_{FX,FY,FZ}"]\ar[d,"\mu_{X,Y}\otimes1_{FZ}"']&
      FX\otimes_\cD (FY\otimes_\cD FZ) \ar[d,"1_{FX}\otimes\mu_{Y,Z}"]\\
      F(X\otimes_\cC Y)\otimes_\cD FZ\ar[d,"\mu_{X\otimes Y,Z}"']&
      FX\otimes_\cD F(Y\otimes_\cC Z)\ar[d,"\mu_{X,Y\otimes Z}"]\\
      F((X\otimes_\cC Y)\otimes_\cC Z)\ar[r,"F(\alpha_{X,Y,Z})"']&
      F(X\otimes_\cC (Y\otimes_\cC Z)),
    \end{tikzcd}
  \end{center}
  \begin{center}
    \begin{tikzcd}
      \I_\cD\otimes_\cD FX\ar[r, "\varepsilon\otimes 1_{FX}"]\ar[d,"\lambda_{FX}"']&
      F(\I_\cC)\otimes_\cD FX\ar[d,"\mu_{\I,X}"]\\
      FX&
      F(\I_\cC\otimes_\cC X),\ar[l, "F(\lambda_X)"]
    \end{tikzcd}
    \hfil
    \begin{tikzcd}
      FX\otimes_\cD \I_\cD\ar[r, "1_{FX}\otimes\varepsilon"]\ar[d,"\rho_{FX}"']&
      FX\otimes_\cD F(\I_\cC)\ar[d,"\mu_{X,\I}"]\\
      FX&
      F(X\otimes_\cC\I_\cC).\ar[l, "F(\rho_X)"]
    \end{tikzcd}
  \end{center}
\end{defn}

There is a dual notion, in which the unit $\varepsilon$ and the
multiplication $\mu$ go in the opposite direction.
\begin{defn}
  An \emph{oplax monoidal functor} from $\cC$ to $\cD$ consists of a
  functor $\cC\rightarrow\cD$ between the underlying categories,
  together with a morphism $\varepsilon:F(\I_\cC)\rightarrow\I_\cD$
  (the \emph{counit}) and a natural transformation
  $\mu_{X,Y}:F(X\otimes_\cC Y)\rightarrow FX\otimes_\cD FY$ (the
  \emph{comultiplication}), such that diagrams commute which are
  identical to those in the previous definition, but where all arrows
  featuring $\varepsilon$ and $\mu$ are reversed.
\end{defn}

An oplax monoidal functor $\cC\rightarrow\cD$ is therefore a lax monoidal
functor $\cC^\op\rightarrow\cD^\op$.

\begin{defn}
  A \emph{strong monoidal functor} from $\cC$ to $\cD$ is a lax
  monoidal functor where $\varepsilon$ and $\mu_{X,Y}$ are
  isomorphisms.
\end{defn}

As well as being a lax monoidal functor, it is quick to check that a
strong monoidal functor also gives an oplax monoidal functor (with
$\varepsilon$ and $\mu_{X,Y}$ replaced by their inverses): it is the
same thing as an oplax monoidal functor whose counit and
comultiplication are isomorphisms.

Composites of lax monoidal functors are lax monoidal functors, those
of oplax monoidal functors are oplax monoidal functors, and those of
strong monoidal functors are strong monoidal functors.

\subsection{Monoidal natural transformations}

Next, we recall monoidal natural transformations.
\begin{defn}
  If $\cC$ and $\cD$ are monoidal categories and
  $F,G:\cC\rightarrow\cD$ lax monoidal functors, then a \emph{monoidal
    natural transformation} $\Theta:F\Rightarrow G$ is a natural
  transformation such that the following diagrams commute:
  \begin{center}
    \begin{tikzcd}
      F(X)\otimes F(Y)\ar[r, "\Theta\otimes\Theta"]\ar[d,"\mu"']&G(X)\otimes G(Y)\ar[d,"\mu"]\\
      F(X\otimes Y)\ar[r,"\Theta"']&G(X\otimes Y),
    \end{tikzcd}
    \hfil
    \begin{tikzcd}
      &\I\ar[dl,"\varepsilon"']\ar[dr,"\varepsilon"]&\\
      F(\I)\ar[rr,"\Theta"']&&G(\I).
    \end{tikzcd}
  \end{center}
\end{defn}

There is an analogous definition of monoidal natural transformation
between oplax monoidal functors: the diagrams are the same except that
$\varepsilon$ and $\mu$ are reversed. There is only one definition of
monoidal natural transformation between strong monoidal functors, and
it does not matter whether we regard them as lax or as oplax.

\subsection{Monoid objects}

Monoidal categories allow us to talk about objects with algebraic and
in particular monoidal structure.
\begin{defn}
  A \emph{monoid object} in a monoidal category $(\cC,\otimes,\I)$ is a
  lax monoidal functor $1\rightarrow\cC$ from the terminal category
  $1$ to $\cC$; a \emph{comonoid object} is an oplax monoidal functor
  $1\rightarrow\cC$.
\end{defn}
This may seem underspecified, but it is easy to see that the terminal
category $1$ has a unique monoidal structure (there is only one choice
of objects and morphisms). Note also that strong monoidal functors
$1\rightarrow\cC$ are of little interest: they simply pick objects
isomorphic to $\I$.

Spelling out this definition, a monoid object consists of an object
$X$ and morphisms $m:X\otimes X\rightarrow X$ (the
\emph{multiplication}) and $u:\I\rightarrow X$ (the \emph{unit}) such
that the following diagrams commute:
\begin{center}
  \begin{tikzcd}
    (X\otimes X)\otimes X\ar[dd, "\alpha"']\ar[r, "m\otimes 1"] & X\otimes X\ar[d, "m"]\\
    &X\\
    X\otimes (X\otimes X)\ar[r, "1\otimes m"'] & X\otimes X,\ar[u, "m"']
  \end{tikzcd}
  \hfil
  \begin{tikzcd}
    \I\otimes X\ar[r, "u\otimes 1"]\ar[dr, "\lambda"']&X\otimes X\ar[d, "m"]&X\otimes\I.\ar[l, "1\otimes u"']\ar[dl, "\rho"]\\
    &X&
  \end{tikzcd}
\end{center}
A \emph{comonoid object} consists of an object $X$, together with
morphisms $X\rightarrow X\otimes X$ (the \emph{comultiplication}) and
$X\rightarrow\I$ (the \emph{counit}), such that similar diagrams
commute but with those featuring $m$ and $u$ reversed.

A simple consequence of the definition is that lax monoidal functors
preserve monoid objects (that is, given a monoid object
$1\rightarrow\cC$ and a lax monoidal functor $\cC\rightarrow\cD$,
their composite is a monoid object in $\cD$), whereas oplax monoidal
functors preserve comonoid objects.

Monoid objects in a monoidal category themselves form a category
(where the morphisms are monoidal natural transformations of functors
$1\rightarrow\cC$). Comonoid objects form a category similarly.

\begin{defn}
  A monoid object $M$ in a symmetric monoidal category  $\cC$ is
  \emph{commutative}  if the diagram
  \begin{center}
    \begin{tikzcd}
      M\otimes M\ar[rr, "\tau"]\ar[dr, "m"']&&M\otimes M\ar[dl, "m"]\\
      &M&
    \end{tikzcd}
  \end{center}
  commutes. Similarly, we may define a \emph{cocommutative} comonoid
  object in such a category, where we reverse the arrows labelled $m$.
\end{defn}


\section{Examples of monoidal categories}
\label{s:examples-monoidal}

\subsection{Cartesian monoidal structure}

We use some basic examples of monoidal categories and functors between
them repeatedly. The following routine construction saves some time.
\begin{example}
  Any category with finite products has the \emph{cartesian monoidal
    structure} with $\otimes=\times$ as product and $\I=1$ as terminal
  object (the empty product).

  Dually, any category with finite coproducts has the
  \emph{cocartesian monoidal structure} with $\otimes=\sqcup$ as
  coproduct and $\I=0$ as initial object (the empty coproduct).
\end{example}
Any functor that preserves finite products extends to a strong
monoidal functor between the cartesian monoidal structures; dually,
any functor that preserves finite coproducts extends to a strong
monoidal functor between the cocartesian monoidal structures.

\begin{example}~ The category $\Set$ of sets and functions has a
  cartesian monoidal structure with $\otimes=\times$ and $\I=1$ (any
  singleton set) and a cocartesian monoidal structure with
  $\otimes=\sqcup$ (the disjoint union of sets) and $\I=\emptyset$. A
  monoid object for the cartesian structure on $\Set$ is just an
  ordinary monoid.
\end{example}

\begin{example}~ The same holds of the category $\Set_*$ of sets and
  partial functions. As the faithful functor $\Set\rightarrow\Set_*$
  preserves finite products and coproducts, it induces a strong
  monoidal functor on both the cartesian and cocartesian monoidal
  structures.

A monoid object for the the cartesian structure on $\Set_*$ is a
\emph{partial monoid}: one where the product is defined only on a
subset of $M\times M$, such that
\begin{itemize}
\item $1a$ and $a1$ are always defined and equal to $a$, and
\item $(ab)c$ is defined if and only if $a(bc)$ is, and whenever
  either is defined they are equal.
\end{itemize}
\end{example}

\subsection{Relational monoids}

The category $\Rel$ of sets and relations has finite products and
coproducts. Both coincide with the disjoint union of sets, so its
cartesian and cocartesian monoidal structures agree. It has another
monoidal structure, the \emph{tensor} monoidal structure, given on
objects by the cartesian product of sets, and on morphisms by
$((x,v),(y,w))\in R\times S$ if and only if $(x,y)\in R$ and $(v,w)\in
S$. The functor $\Set_*\rightarrow\Rel$ sending a partial function to
its graph is then strong monoidal.

A monoid in $\Rel$ (with the tensor monoidal structure) -- a
\emph{relational monoid} -- consists of a set $X$, a ``unit'' relation
from $1$ to $X$ (we write $U^x$ if the unique element of $1$ is
related to $x$) and a ternary ``multiplication'' relation, thought of
as a relation from $X\times X$ to $X$ (we write $M^x_{yz}$ if $(y,z)$
is related to $x$). These data are subject to the relational unitality
conditions
\[\bigvee_{a\in X}(U^a\wedge M^x_{ay}) \Longleftrightarrow \bigvee_{a\in
  X}(M^x_{ya}\wedge U^a) \Longleftrightarrow x=y\] for all $x,y\in X$,
and the relational associativity condition
\[\bigvee_{a\in X}(M^w_{xa}\wedge M^a_{yz}) \Longleftrightarrow \bigvee_{b\in X}(M^b_{xy}\wedge M^w_{bz})\]
for all $w,x,y,z\in X$, as considered, for instance, in
\cites{KenPar,Rosenthal}.

Here, the notational choice of whether to write indices as subscripts
or superscripts has not been made arbitrarily: our algebraic
operations, when regarded as relations, have variables corresponding
to inputs as subscripts and those corresponding to outputs as
superscripts. This choice evokes the Einstein summation from
physics. All sums are quantified over variables that appear once as an
output and once as an input, and the semantics of this summation is
given by composition.

The category $\Rel$ is symmetric monoidal; a commutative monoid object
(a \emph{commutative relational monoid}), satisfies the additional
relation $M_{xy}^z \Longleftrightarrow M_{yx}^z$ for all $x,y,z\in X$.

If $X$ is a relational monoid with multiplication $M$ and unit $U$,
and $Y$ a relational monoid with multiplication $N$ and unit $V$,
then a \emph{homomorphism} $X\rightarrow Y$ of relational monoids is a
relation $F$ from $X$ to $Y$ which is compatible with multiplication
in the sense that
\[\bigvee_{c\in X}(M_{ab}^c\wedge F_c^z)\Longleftrightarrow \bigvee_{x,y\in Y}(F_a^x\wedge F_b^y\wedge N_{xy}^z),\]
and compatible with units in the sense that
\[\bigvee_{a\in X}(U^a\wedge F_a^x)\Longleftrightarrow V^x.\]
This agrees with the definition in \cite{Rosenthal}*{Definition 1.5}.

If $F$ is the graph of a function $f$, so that $F_a^x$ if and only if
$x=f(a)$, then the conditions reduce to
\[\bigvee_{f(c)=z}M_{ab}^c\Longleftrightarrow
  N_{f(a)f(b)}^z\qquad\text{ and }\qquad
  \bigvee_{f(a)=x}U^a\Longleftrightarrow V^x.\]

There are functors $\Set\rightarrow\Set_*\rightarrow\Rel$, which in
turn regard total functions as a subclass of partial functions, and
partial functions as a subclass of relations (via their graphs). These
compare the monoid objects in these three categories: a monoid is a
special case of a partial monoid, which is in turn a special case of a
relational monoid.

\begin{example}
\label{ex:shuffle-relational}
A well-studied example of a commutative relational monoid is given by
shuffles of words. Let $A$ be a set (the \emph{alphabet}), and let
$M=A^*$ be the free monoid on $A$ (the set of \emph{words} on
$A$).

The ternary \emph{shuffle} relation $\shuffle$ on $M$, which we write
as a multioperation $M\times M\rightarrow\bP M$, is defined
recursively by considering when either argument is the empty word $e$
or when both words are nonempty with first letters $x$ and $y$:
\begin{align*}
  u\shuffle e = e\shuffle u &= \{u\},\\
  xu\shuffle yv &= \{xw\mid w\in u\shuffle yv\} \cup \{yw\mid w\in xu\shuffle v\}.
\end{align*}
This ternary relation, with the subset of units containing the empty
word only, gives $M$ indeed the structure of a commutative relational
monoid (with a single unit). We will return to this example in
Sections~\ref{ss:relational} and \ref{ss:asb} below.
\end{example}

\begin{remark}
Examples of relational monoids and related structures can be found in
the semantics of substructural logics. Kripke frames based on ternary
relations are used for instance in relevance
logics~\cite{DunnRestall02}, the Lambek calculus~\cite{Lambek58},
categorical logics~\cite{MootR12} or linear
logic~\cite{AllweinD93}. More generally, $(n+1)$-ary relational
structures provide semantics for boolean algbras with $n$-ary (modal)
operators~\cite{JonTar}, of which the ternary Kripke frames mentioned
form instances. In this context, functional notions of homomorphisms
between relational structures are used: so-called \emph{$p$-morphisms}
or \emph{bounded homomorphisms}~\cite{Segerberg,Goldblatt}. Bounded
homomorphisms are the appropriate notion of morphism in the category
of ternary relational structures that is dual to the category of
boolean algebras with binary operators~\cite{Goldblatt}. Rosenthal as
proved essentially the same duality between relational monoids on $X$
and categories of powerset quantales on $X$ using relational
homomorphisms~\cite{Rosenthal}.  However, the functional
specialisation of the relational notion of homomorphism between
relational monoids is neither a special case nor a generalisation of
the notion of bounded homomorphism.
\end{remark}

\section{$n$-fold monoidal categories}
\label{s:n-fold-monoidal}

We define $n$-fold monoidal categories as categories with $n$
different interacting mo\-noidal structures. These are roughly as
introduced in \cite{BFSV}, but we require extra generality in two
ways: in our examples we cannot ensure the same level of strictness,
and we wish to allow different units for the $n$ different
structures. This concept of $n$-fold monoidal category was also
introduced, for different purposes, in \cite{AguMah}.

\subsection{Double monoidal categories}

We start by defining double monoidal categories, since these will come
up most commonly, and form a good building block for the general
construction.
\begin{defn}
  A \emph{double monoidal category} is a category $\cC$ with two
  monoidal structures $(\otimes, \I, \alpha, \lambda, \rho)$ and
  $(\odot, \1, a, l, r)$, such that the latter consists of lax
  monoidal functors and monoidal natural transformations with respect
  to the former.
\end{defn}

We spell out this definition in terms of commutative diagrams. For
$\odot:\cC^2\rightarrow\cC$ to be lax means that there is a map
\begin{equation*}
\chi^{W,X}_{Y,Z}:(W\odot X)\otimes(Y\odot Z)\longrightarrow(W\otimes
Y)\odot(X\otimes Z),
\end{equation*}
the multiplication, and a map $\zeta:\I\rightarrow\I\odot\I$, the
unit.

These make the usual three diagrams for a lax monoidal functor
commute:
\begin{equation}\tag{D1}
  \begin{tikzcd}
    ((U\odot V)\otimes(W\odot X))\otimes(Y\odot Z)\ar[r,"\alpha"]\ar[d,"\chi\otimes 1"']&
    (U\odot V)\otimes((W\odot X)\otimes(Y\odot Z))\ar[d,"1\otimes\chi"]\\
    ((U\otimes W)\odot(V\otimes X))\otimes(Y\odot Z)\ar[d,"\chi"']&
    (U\odot V)\otimes((W\otimes Y)\odot(X\otimes Z))\ar[d,"\chi"]\\
    ((U\otimes W)\otimes Y)\odot((V\otimes X)\otimes Z)\ar[r,"\alpha\odot\alpha"']&
    (U\otimes(W\otimes Y))\odot(V\otimes(X\otimes Z)),
  \end{tikzcd}
\end{equation}
\begin{equation}\tag{D2}
  \begin{tikzcd}
    \I\otimes(X\odot Y)\ar[r, "\zeta\odot 1"]\ar[d,"\lambda"']&
    (\I\odot\I)\otimes(X\odot Y)\ar[d,"\chi"]\\
    X\odot Y&
    (\I\otimes X)\odot (\I\otimes Y),\ar[l, "\lambda\odot\lambda"]
  \end{tikzcd}
\end{equation}
\begin{equation}\tag{D3}
  \begin{tikzcd}
    (X\odot Y)\otimes\I\ar[r, "1\odot\zeta"]\ar[d,"\rho"']&
    (X\odot Y)\otimes(\I\odot\I)\ar[d,"\chi"]\\
    X\odot Y&
    (X\otimes\I)\odot(Y\otimes\I).\ar[l, "\rho\odot\rho"]
  \end{tikzcd}
\end{equation}

For $\1:1\rightarrow\cC$ to be lax means that there are maps
$\nu:\1\otimes\1\rightarrow\1$ (the multiplication) and
$\iota:\I\rightarrow\1$ (the unit). These are such that the usual
three diagrams commute. In this case this reduces to the following:
\begin{equation}\tag{D4}
  \begin{tikzcd}
    (\1\otimes\1)\otimes\1\ar[rr,"\alpha"]\ar[d,"\nu\otimes 1"']&&
    \1\otimes(\1\otimes\1)\ar[d,"1\otimes\nu"]\\
    \1\otimes\1\ar[r,"\nu"']&
    \1&
    \1\otimes\1,\ar[l,"\nu"]\\
  \end{tikzcd}
\end{equation}
\begin{equation}\tag{D5, 6}
  \begin{tikzcd}
    \I\otimes\1\ar[rr,"\iota\odot 1"]\ar[dr, "\lambda"']&&
    \1\otimes\1,\ar[dl,"\nu"]\\
    &\1
  \end{tikzcd}
  \hfil
  \begin{tikzcd}
    \1\otimes\I\ar[rr,"1\odot \iota"]\ar[dr, "\rho"']&&
    \1\otimes\1.\ar[dl,"\nu"]\\
    &\1
  \end{tikzcd}
\end{equation}

For $l$ to be a monoidal natural transformation, the following two
diagrams must commute:
\begin{equation}\tag{D7, 8}
  \begin{tikzcd}
    (\1\otimes\1)\odot(X\otimes Y)\ar[d,"\nu\odot1"']&(\1\odot X)\otimes(\1\odot Y)\ar[l, "\chi"']\ar[d,"l\odot l"]\\
    \1\odot(X\otimes Y)\ar[r,"l"']&X\otimes Y,
  \end{tikzcd}
  \hfil
  \begin{tikzcd}
    \I\odot\I\ar[d,"\iota\odot1"']&\I\ar[l,"\zeta"']\ar[d,"1"]\\
    \1\odot\I\ar[r,"l"']&\I.
  \end{tikzcd}
\end{equation}
This uses the lax monoidal structure on the functor
$\1\odot-:\cC\rightarrow\cC$, which is the composite of $\odot$
with $\1\times 1$.

Similarly, for $r$ to be monoidal, the following diagrams must
commute:
\begin{equation}\tag{D9, 10}
  \begin{tikzcd}
    (X\otimes Y)\odot(\1\otimes\1)\ar[d,"1\odot\nu"']&(X\odot\1)\otimes(Y\odot\1)\ar[l, "\chi"']\ar[d,"r\otimes r"]\\
    (X\otimes Y)\odot\1\ar[r,"r"']&X\otimes Y,
  \end{tikzcd}
  \hfil
  \begin{tikzcd}
    \I\odot\I\ar[d,"1\odot\iota"']&\I\ar[l,"\zeta"']\ar[d,"1"]\\
    \I\odot\1\ar[r,"r"']&\I.
  \end{tikzcd}
\end{equation}

Lastly, for $a$ to be monoidal, the following diagrams must commute:
\begin{equation}\tag{D11}
  \begin{tikzcd}
    ((U\odot V)\odot W)\otimes((X\odot Y)\odot Z)\ar[r,"a\otimes a"]\ar[d,"\chi"']&
    (U\odot(V\odot W))\otimes(X\odot(Y\odot Z))\ar[d,"\chi"]\\
    ((U\odot V)\otimes(X\odot Y))\odot(W\otimes Z)\ar[d,"\chi\odot1"']&
    (U\otimes X)\odot((V\odot W)\otimes (Y\odot Z))\ar[d,"1\odot\chi"]\\
    ((U\otimes X)\odot(V\otimes Y))\odot(W\otimes Z)\ar[r,"a"']&
    (U\otimes X)\odot((V\otimes Y)\odot(W\otimes Z)),
  \end{tikzcd}
\end{equation}
\begin{equation}\tag{D12}
  \begin{tikzcd}
    \I\odot\I\ar[d,"\zeta\odot1"']&\I\ar[l,"\zeta"']\ar[r,"\zeta"]&\I\odot\I\ar[d,"1\odot\zeta"]\\
    (\I\odot\I)\odot\I\ar[rr,"a"']&&\I\odot(\I\odot\I).
  \end{tikzcd}
\end{equation}

These use the monoidal structures on the two functors
$\cC^3\rightarrow\cC$ sending $(X,Y,Z)$ to $(X\odot Y)\odot Z$ and
to $X\odot(Y\odot Z)$.

\begin{remark}
  The diagrams D1 and D11 are elaborations of the diagrams from the
  definition of a $2$-fold monoidal category (in the strictly
  associative and unital setting) in \cite{BFSV}*{Definition 1.3}.
\end{remark}

There is a duality at work: each of the twelve diagrams above has a
companion where the roles of the two monoidal structures are
exchanged. More specifically, this duality exchanges D1 with D11, D2
with D7, D3 with D9, D4 with D12, D5 with D7 and D6 with D10. This
leads to an alternative, equivalent definition of a double monoidal
category as a category with a pair of monoidal structures such that
$(\otimes, \I, \alpha, \lambda, \rho)$ consists of oplax monoidal
functors and monoidal natural transformations with respect to
$(\odot, \1, a, l, r)$.

We refer to these diagrams, which relate two different monoidal
structures, jointly as \emph{double interchange diagrams}.

\begin{remark}
  In many examples below, the units $\1$ and $\I$ agree, $\iota$ is
  the identity and $\zeta$ and $\nu$ are unit isomorphisms. Most of
  the diagrams above then commute automatically; only D1 and D11 need
  checking.
\end{remark}

\begin{remark}
  It can be read off the definition that $\1$ is a monoid object in
  $(\cC,\otimes,\I)$ and $\I$ is a comonoid object in
  $(\cC,\odot,\1)$.
\end{remark}

\begin{remark}
  Aguiar and Mahajan \cite{AguMah}*{Definition 6.1} give an equivalent
  definition, somewhat differently motivated and presented. Their
  (6.3) is our D1, their (6.4) our D11, their (6.5) our D2/D3, their
  (6.6) our D7/D9, their (6.7) our D4/D5/D6, and their (6.8) our
  D8/D10/D12.
\end{remark}

\begin{remark}
  The four structure maps share a pattern that is useful for
  explaining constructions below. All of them are maps
\[\bigotimes_{a\in A}\bigodot_{b\in B}X_{ab}\longrightarrow\bigodot_{b\in B}\bigotimes_{a\in A}X_{ab}.\]
For $\chi$, we have $|A|=|B|=2$; for $\zeta$, we have $A=\emptyset$
and $|B|=2$; for $\nu$, we have $|A|=2$ and $B=\emptyset$; for
$\iota$, we have $A=B=\emptyset$.
\end{remark}

The four structure maps are not entirely independent:
\begin{prop}
   The map $\iota$ is determined by $\chi$ as the following composite:
\[\I\longrightarrow\I\otimes\I\longrightarrow(\I\odot\1)\otimes(\1\odot\I)\stackrel{\chi}{\longrightarrow}(\I\otimes\1)\odot(\1\otimes\I)\longrightarrow\1\odot\1\longrightarrow\1.\]
\end{prop}
\begin{proof}
The following commutative diagram shows this (for simplicity, we trim
the morphism labels to remove identities, associators and
unitors).

The central square is an instance of D9; the top quadrilateral a
naturality square for $\chi$, and the thin triangle on the right an
instance of D6:
\begin{displaymath}
  \begin{tikzcd}
    & (\I\odot\1)\otimes(\1\odot\I)\ar[rr, "\chi"]\ar[d, "\iota"'] && (\I\otimes\1)\odot(\1\otimes\I)\ar[dl, "\iota"]\ar[dd]\ar[ddl]\\
    \I\otimes\I\ar[ur] & (\I\odot\1)\otimes(\1\odot\1)\ar[r, "\chi"]\ar[d] & (\I\otimes\1)\odot(\1\otimes\1)\ar[d] &\\
    & \I\otimes\1 & (\I\otimes\1)\odot\1\ar[l]\ar[dr] & \1\odot\1\ar[d]\\
    \I\ar[uu]\ar[ur]\ar[rrr, "\iota"'] &&& \1
  \end{tikzcd}
\end{displaymath}
\end{proof}

We put off giving examples until Section \ref{s:examples-nfold}, since
the majority of examples of double monoidal categories are in fact
$n$-fold monoidal categories, a concept which we now go on to define.

\subsection{$n$-fold monoidal categories}

Now we define our main object of interest: categories with more than
two interacting monoidal structures.

\begin{defn}
  An \emph{$n$-fold monoidal category} is a category $\cC$ equipped with:
  \begin{itemize}
  \item $n$ different monoidal structures
    $(\otimes_i,\I_i,\alpha_i,\lambda_i,\rho_i)$ for each $1\leq
    i\leq n$;
  \item a double monoidal structure relating the $i$th and the $j$th
    monoidal structure for each $1\leq i< j\leq n$, giving maps
    \begin{itemize}
    \item $\chi_{ij}:(W\otimes_jX)\otimes_i(Y\otimes_jZ) \rightarrow
      (W\otimes_iY)\otimes_j(X\otimes_iZ)$,
    \item $\zeta_{ij}:\I_i\rightarrow\I_i\otimes_j\I_i$,
    \item $\nu_{ij}:\I_j\otimes_i\I_j\rightarrow\I_j$,
    \item $\iota_{ij}:\I_i\rightarrow\I_j$;
    \end{itemize}
  \end{itemize}
  such that, for every $1\leq i<j<k\leq n$, the natural
  transformations $\chi_{jk}$, $\zeta_{jk}$, $\nu_{jk}$ and
  $\iota_{jk}$ are monoidal natural transformations (between lax
  monoidal functors) with respect to $\otimes_i$.
\end{defn}

This definition can be expanded to say that, as well as the double
interchange diagrams holding, the following eight diagrams (the
\emph{triple interchange diagrams}) all commute for every $1\leq
i<j<k\leq n$:
  \begin{equation}\tag{T1}
    \begin{tikzcd} 
      \begin{tabular}{C}
        ((S\otimes_kT)\otimes_j(U\otimes_kV))\\
        \otimes_i\\
        ((W\otimes_kX)\otimes_j(Y\otimes_kZ))
      \end{tabular}
      \ar[r, "\chi_{jk}\otimes_i\chi_{jk}"]
      \ar[d, "\chi_{ij}"']&
      \begin{tabular}{C}
        ((S\otimes_jU)\otimes_k(T\otimes_jV))\\
        \otimes_i\\
        ((W\otimes_jY)\otimes_k(X\otimes_jZ))
      \end{tabular}
      \ar[d, "\chi_{ik}"]\\
      \begin{tabular}{C}
        ((S\otimes_kT)\otimes_i(W\otimes_kX))\\
        \otimes_j\\
        ((U\otimes_kV)\otimes_i(Y\otimes_kZ))
      \end{tabular}
      \ar[d, "\chi_{ik}\otimes_j\chi_{ik}"']&
      \begin{tabular}{C}
        ((S\otimes_jU)\otimes_i(W\otimes_jY))\\
        \otimes_k\\
        ((T\otimes_jV)\otimes_i(X\otimes_jZ))
      \end{tabular}
      \ar[d, "\chi_{ij}\otimes_k\chi_{ij}"]\\
      \begin{tabular}{C}
        ((S\otimes_iW)\otimes_k(T\otimes_iX))\\
        \otimes_j\\
        ((U\otimes_iY)\otimes_k(V\otimes_iZ))
      \end{tabular}
      \ar[r, "\chi_{jk}"']
      &
      \begin{tabular}{C}
        ((S\otimes_iW)\otimes_j(U\otimes_iY))\\
        \otimes_k\\
        ((T\otimes_iX)\otimes_j(V\otimes_iZ)),
      \end{tabular}
    \end{tikzcd}
  \end{equation}
  \begin{equation}\tag{T2}
    \begin{tikzcd} 
      (\I_k\otimes_j\I_k)\otimes_i(\I_k\otimes_j\I_k)
      \ar[d, "\nu_{jk}\otimes_i\nu_{jk}"']
      \ar[rr, "\chi_{ij}"]&&
      (\I_k\otimes_i\I_k)\otimes_j(\I_k\otimes_i\I_k)
      \ar[d, "\nu_{ik}\otimes_j\nu_{ik}"]\\
      \I_k\otimes_i\I_k
      \ar[r, "\nu_{ik}"']&
      \I_k&
      \I_k\otimes_j\I_k,
      \ar[l, "\nu_{jk}"]
    \end{tikzcd}
  \end{equation}
  \begin{equation}\tag{T3}
    \begin{tikzcd} 
      \I_i\otimes_j\I_i
      \ar[d, "\zeta_{ik}\otimes_j\zeta_{ik}"'] &
      \I_i
      \ar[l, "\zeta_{ij}"']
      \ar[r, "\zeta_{ik}"]&
      \I_i\otimes_k\I_i
      \ar[d, "\zeta_{ij}\otimes_k\zeta_{ij}"]\\
      (\I_i\otimes_k\I_i)\otimes_j(\I_i\otimes_k\I_i)
      \ar[rr, "\chi_{jk}"']&&
      (\I_i\otimes_j\I_i)\otimes_k(\I_i\otimes_j\I_i),
    \end{tikzcd}
  \end{equation}
  \begin{equation}\tag{T4}
    \begin{tikzcd} 
      (\I_j\otimes_k\I_j)\otimes_i(\I_j\otimes_k\I_j)
      \ar[rr, "\chi_{ik}"]&&
      (\I_j\otimes_i\I_j)\otimes_k(\I_j\otimes_i\I_j)
      \ar[d, "\nu_{ij}\otimes_k\nu_{ij}"]\\
      \I_j\otimes_i\I_j
      \ar[u, "\zeta_{jk}\otimes_i\zeta_{jk}"]
      \ar[r, "\nu_{ij}"']&
      \I_j
      \ar[r, "\zeta_{jk}"']&
      \I_j\otimes_k\I_j,
    \end{tikzcd}
  \end{equation}
  \begin{equation}\tag{T5, 6}
    \begin{tikzcd} 
      \I_j\otimes_i\I_j
      \ar[r, "\iota_{jk}\otimes_i\iota_{jk}"]
      \ar[d, "\nu_{ij}"']&
      \I_k\otimes_i\I_k
      \ar[d, "\nu_{ik}"]\\
      \I_j
      \ar[r, "\iota_{jk}"']&
      \I_k,
    \end{tikzcd}
    \hfil
    \begin{tikzcd} 
      \I_i
      \ar[r, "\iota_{ij}"]
      \ar[d, "\zeta_{ik}"']&
      \I_j
      \ar[d, "\zeta_{jk}"]\\
      \I_i\otimes_k\I_i
      \ar[r, "\iota_{ij}\otimes_k\iota_{ij}"']&
      \I_j\otimes_k\I_j,
    \end{tikzcd}
  \end{equation}
  \begin{equation}\tag{T7, 8}
    \begin{tikzcd} 
      \I_i
      \ar[r, "\zeta_{ij}"]
      \ar[d, "\iota_{ik}"']&
      \I_i\otimes_j\I_i
      \ar[d, "\iota_{ik}\otimes_j\iota_{ik}"]\\
      \I_k&
      \I_k\otimes_j\I_k,
      \ar[l, "\nu_{jk}"]
    \end{tikzcd}
    \hfil
    \begin{tikzcd} 
      &\I_j\ar[dr, "\iota_{jk}"]&\\
      \I_i\ar[ur, "\iota_{ij}"]\ar[rr, "\iota_{ik}"']&&\I_k.
    \end{tikzcd}
  \end{equation}
  Of these, (T1) and (T3) express that the natural transformation $\chi_{jk}$ is
  monoidal, (T4) and (T6) that $\zeta_{jk}$ is monoidal, (T2) and (T7)
  that $\nu_{jk}$ is monoidal, and (T5) and (T8) that $\iota_{jk}$ is
  monoidal.

\begin{remark}
  Diagram T1 is exactly the diagram seen in the definition of an
  $n$-fold monoidal category (in the strictly associative and unital
  setting) seen in \cite{BFSV}*{Definition 1.7}.
\end{remark}

\begin{remark}
  Our definition of $n$-fold monoidal categories is equivalent to
  Aguiar and Mahajan's \cite{AguMah}*{Definitions 7.1, 7.24}: their
  (7.3) is our T1, their (7.4) our T3, their (7.5) our T4, their (7.6)
  our T2, their (7.7) our T6/T5/T7 and their (7.8) our T8.
\end{remark}

\begin{remark}
There is a pattern to these eight diagrams: each is a simplified
hexagon providing an equality between two maps
\[\bigotimes^{(i)}_{a\in A}\bigotimes^{(j)}_{b\in B}\bigotimes^{(k)}_{c\in C}X_{abc}
    \longrightarrow
  \bigotimes_{c\in C}^{(k)}\bigotimes_{b\in B}^{(j)}\bigotimes_{a\in A}^{(i)}X_{abc},\]
resulting from the two different ways of reversing the order of the
tensors by swapping adjacents, two at a time.

These eight diagrams state that, given a $k$-product of a family of
$j$-products of families of $i$-products, the results are the same no
matter in which order we reverse these three. There is one diagram
reflecting each possible choice of $\otimes_i$ or $\I_i$, $\otimes_j$
or $\I_j$, and $\otimes_k$ or $\I_k$. Only if we choose the
multiplication in all three cases do we get a diagram with any
variables in it!
\end{remark}

\begin{remark}
  Alternatively, we could have defined an $n$-fold monoidal category
  by requiring that, for each $1\leq i<j<k\leq n$, $\chi_{ij}$,
  $\zeta_{ij}$, $\nu_{ij}$ and $\iota_{ij}$ are monoidal natural
  transformations (between oplax monoidal functors) with respect to
  $\otimes_k$. Then (T1) and (T2) express that $\chi_{ij}$ is
  monoidal, (T3) and (T7)  that $\zeta_{ij}$ is monoidal, (T4)
  and (T5) that $\nu_{ij}$ is monoidal, and (T6) and (T8)
   that $\iota_{ij}$ is monoidal.
\end{remark}

\begin{remark}
  It is clear that an $n$-fold monoidal category with $n=1$ is a
  monoidal category; with $n=2$ it is a double monoidal category.
\end{remark}

\begin{remark}
  Monoidal structures sometimes share units, and the maps $\iota$,
  $\zeta$ and $\nu$ are sometimes the obvious isomorphisms. Where this
  happens, only D1 and D11 need to be checked for each pair of
  monoidal structures and T1 is the only triple interchange diagram
  that is not automatic.
\end{remark}

\begin{conj}
  We conjecture that there may be a coherence theorem, à la Mac Lane,
  for this definition. Presumably it should say that all maps
\[\bigotimes^{(1)}_{a_1\in A_1}\cdots\bigotimes^{(n)}_{a_n\in A_n}X_{a_1,\ldots,a_n}
  \longrightarrow
  \bigotimes^{(n)}_{a_n\in A_n}\cdots\bigotimes^{(1)}_{a_1\in A_1}X_{a_1,\ldots,a_n}\]
defined as composites of associators, unitors, and interchange
morphisms are equal.
\end{conj}

\begin{remark}
  If $\cC$ has an $n$-fold monoidal structure, then $\cC^\op$ has one
  as well, and opposition reverses the order of the monoidal
  structures: $(\otimes_i,\I_i)$ on $\cC^\op$ is obtained from
  $(\otimes_{n+1-i},\I_{n+1-i})$ on $\cC$.
\end{remark}

\subsection{Functors of $n$-fold monoidal categories}

Having defined these highly structured categories, one should also
define the appropriate notion of functors between them. There is a
bewildering array of options. In fact, a definition is required to
specify the various concepts:
\begin{defn}
  An \emph{$n$-fold functor type} $T$ is a list of $n$
  words, $T_1,\ldots,T_n$, each of which is ``lax'', ``oplax'', or
  ``strong'', subject to the condition that there is no $i<j$ with
  $T_i=\text{``oplax''}$ and $T_j=\text{``lax''}$.
\end{defn}

\begin{defn}
  Let $\cC$, $\cD$ be
  $n$-fold monoidal categories and $T$ an $n$-fold functor type.
  A \emph{$n$-fold monoidal functor of type $T$} from $\cC$ to $\cD$
  is a functor $F:\cC\rightarrow\cD$, equipped with a monoidal structure
  of type $T_i$ from $(\cC,\otimes_i,\I_i)$ to $(\cD,\otimes_i,\I_i)$,
  such that the following diagrams commute:
  \begin{center}
    \begin{tikzcd}
      (FW\otimes_jFX)\otimes_i(FY\otimes_jFZ)\ar[r, "\chi"]\ar[d, "\mu\otimes\mu"', slashed]&
      (FW\otimes_iFY)\otimes_j(FX\otimes_iFZ)\ar[d, "\mu\otimes\mu", double slashed]\\
      F(W\otimes_jX)\otimes_iF(Y\otimes_jZ)\ar[d, "\mu"', double slashed]&
      F(W\otimes_iY)\otimes_jF(X\otimes_iZ)\ar[d, "\mu", slashed]\\
      F((W\otimes_jX)\otimes_i(Y\otimes_jZ))\ar[r, "F(\chi)"']&
      F((W\otimes_iY)\otimes_j(X\otimes_iZ)),
    \end{tikzcd}
  \end{center}

  \begin{center}
    \begin{tikzcd}
      \I_i\ar[r, "\zeta"]\ar[dd, "\varepsilon"', double slashed]&
      \I_i\otimes_j\I_i\ar[d, "\varepsilon\otimes\varepsilon", double slashed]\\
      &F(\I_i)\otimes_jF(\I_i)\ar[d, "\mu", slashed]\\
      F(\I_i)\ar[r, "F(\zeta)"']&
      F(\I_i\otimes_j\I_i),
    \end{tikzcd}
    \hfil
    \begin{tikzcd}
      \I_j\otimes_i\I_j\ar[r, "\nu"]\ar[d, "\varepsilon\otimes\varepsilon"', slashed]&
      \I_j\ar[dd, "\varepsilon", slashed]\\
      F(\I_j)\otimes_iF(\I_j)\ar[d, "\mu"', double slashed]&\\
      F(\I_j\otimes_i\I_j)\ar[r, "F(\nu)"']&
      F(\I_j),
    \end{tikzcd}
  \end{center}

  \begin{center}
    \begin{tikzcd}
      \I_i\ar[r, "\iota"]\ar[d, "\varepsilon"', double slashed]&
      \I_j\ar[d, "\varepsilon", slashed]\\
      F(\I_i)\ar[r, "F(\iota)"']&
      F(\I_j).
    \end{tikzcd}
  \end{center}
\end{defn}
Here a single slash across an arrow indicates that it must be reversed
if $T_j$ is ``oplax'', and a double slash indicates that it must be
reversed if $T_i$ is ``oplax''. What we do if either is ``strong'' is
effectively a matter of preference, since the relevant morphisms are
isomorphisms.

When the functor type is ``lax'' everywhere, we refer to it as
\emph{fully lax}, and similarly for \emph{fully strong} or \emph{fully
  oplax}.

We can take monoid objects as a special case.
\begin{defn}\label{def:n-monoid}
  Let $T$ be an $n$-fold functor type and $\cC$ an $n$-fold monoidal
  category. A \emph{monoid object of type $T$} in $\cC$ is a monoidal
  functor of type $T$ from the terminal category $1$ to $\cC$ (note
  that, as before, $1$ has a unique $n$-fold monoidal category
  structure).

  If $T$ is fully lax, we call the object an \emph{$n$-monoid}, and if
  $T$ fully oplax, we call it an \emph{$n$-comonoid}.
\end{defn}

Here, as with monoid objects in monoidal categories, we are not much
interested in the cases where some $T_i$ is ``strong'' since that
simply forces a monoid object to be isomorphic to $\I_i$.

In general, such an object consists of an object $X$, equipped with
morphisms $X\otimes_i X\rightarrow X$ and $\I_i\rightarrow X$ whenever
$T_i$ is ``lax'', and their reverses whenever $T_i$ is ``oplax'', such
that $X$ is a monoid object or comonoid object according to whether
$T_i$ is ``lax'' or ``oplax'', and such that diagrams commute as
follows:
\begin{center}
  \begin{tikzcd}
    (X\otimes_jX)\otimes_i(X\otimes_jX)\ar[dd, "\chi"']\ar[r, "m\otimes m", slashed]&
    X\otimes_iX\ar[d, "m", double slashed]\\
    &X\\
    (X\otimes_iX)\otimes_j(X\otimes_iX)\ar[r, "m\otimes m", double slashed]&
    X\otimes_jX,\ar[u, "m"', slashed]\\
  \end{tikzcd}
  \hfil
  \begin{tikzcd}
    \I_i\ar[rr, "\iota"]\ar[dr, "u"', double slashed]&&
    \I_j\ar[dl, "u", slashed]\\
    &X,&
  \end{tikzcd}
\end{center}
\begin{center}
  \begin{tikzcd}
    \I_i\ar[r, "\zeta"]\ar[d, "u"', double slashed]&
    \I_i\otimes_j\I_i\ar[d, "u\otimes u", double slashed]\\
    X&X\otimes_jX,\ar[l, "m", slashed]
  \end{tikzcd}
  \hfil
  \begin{tikzcd}
    \I_j\otimes_i\I_j\ar[r, "\nu"]\ar[d, "u\otimes u"', slashed]&
    \I_j\ar[d, "u", slashed]\\
    X\otimes_iX\ar[r, "m"', double slashed]&
    X.
  \end{tikzcd}
\end{center}
Again we use the convention that single-slashed arrows get reversed if
$T_j$ is ``oplax'', and double-slashed arrows if $T_i$ is ``oplax''.

\begin{defn}
A natural transformation between functors of type $T$ is a natural
transformation that restricts to a monoidal natural transformation on
each monoidal structure.
\end{defn}

This definition means that we can define a category of functors of
type $T$ between two $n$-fold monoidal categories, and in particular a
category of monoid objects of type $T$.

\subsection{Induced maps}
\label{ss:induced-maps}

In examples, there are often canonical maps between different
products. For example, if $\odot$ represents the sequential
composition and $\otimes$ the parallel composition of concurrent
programs or components, then we might expect some models to have
canonical maps $X\odot Y\rightarrow X\otimes Y$ and
$Y\odot X\rightarrow X\otimes Y$, representing that parallel programs
$X$ and $Y$ could as well be scheduled to run sequentially in any
order, or that a specification with parallel components $X$ and $Y$
can be refined to specifications that put them in sequence.

These maps can be described uniformly requiring only a small piece of
extra structure; examples are provided later once we have given
examples of $n$-fold monoidal categories.

Indeed, let $\cC$ be an $n$-fold monoidal category and let
$a:\I_j\rightarrow\I_i$ be a map between units (where $i<j$ as
usual). We can then define natural transformations
$\varphi(a):X\otimes_iY\rightarrow X\otimes_jY$ and
$\bar\varphi(a):X\otimes_iY\rightarrow Y\otimes_j X$ as the following
composites:
\begin{center}
  \begin{tikzcd}
    X\otimes_iY\ar[d]\\
    (X\otimes_j\I_j)\otimes_i(\I_j\otimes_jY)\ar[d, "\chi"']\\
    (X\otimes_i\I_j)\otimes_j(\I_j\otimes_iY)\ar[d, "(1\otimes_ia)\otimes_j(a\otimes_i1)"']\\
    (X\otimes_i\I_i)\otimes_j(\I_i\otimes_iY)\ar[d]\\
    X\otimes_jY,
  \end{tikzcd}
  \hfil
  \begin{tikzcd}
    X\otimes_iY\ar[d]\\
    (\I_j\otimes_jX)\otimes_i(Y\otimes_j\I_j)\ar[d, "\chi"]\\
    (\I_j\otimes_iY)\otimes_j(X\otimes_i\I_j)\ar[d, "(a\otimes_i1)\otimes_j(1\otimes_ia)"]\\
    (\I_i\otimes_iY)\otimes_j(X\otimes_i\I_i)\ar[d]\\
    Y\otimes_jX.
  \end{tikzcd}
\end{center}

\begin{remark}
  These morphisms are familiar from the Eckmann-Hilton argument, but
  need not be invertible in this context.
\end{remark}

\subsection{Lax $n$-fold monoid objects}
\label{ss:lax-monoid-objs}

The $n$-fold monoid objects defined above have strict interchange,
unlike $n$-fold monoidal categories themselves. We can produce a lax
notion, but the morphisms needed for the interchange law require
monoidal bicategories rather than just monoidal categories. As a
monoid can be defined to be a one-object category, and a monoidal
category to be a one-object bicategory \cite{BaeDol1},
\cite{JohYau}*{11.2}, a monoidal bicategory can be defined to be a
one-object tricategory \cite{Gurski}.

This definition is intended so that, for example, lax $n$-fold monoids
in $\Cat$ (regarded as a monoidal bicategory using the cartesian
monoidal structure) are exactly $n$-fold monoidal categories.

\begin{defn}
  Let $\cB$ be a monoidal bicategory. A \emph{lax double monoid
  object} in $\cB$ consists of an object in $X\in\cB$, two monoid
  structures $(X,\otimes,\I)$ and $(X,\odot,\1)$ sharing it as
  underlying object, and 2-cells $\chi,\zeta,\nu,\iota$ between their
  composites which satisfy the double interchange diagrams.
\end{defn}

\begin{defn}
  Let $\cB$ be a monoidal bicategory. A \emph{lax $n$-fold monoid
  object} in $\bB$ consists of an object $X\in\cB$ with $n$ monoid
  structures on it, such that any two of them form a lax double
  monoid object, and such that any three of them also satisfy the
  triple interchange diagrams.
\end{defn}

\begin{remark}
  We do not make use of it in the present work, but a \emph{lax
  $n$-fold comonoid object} can be defined dually, or more immediately
  as a lax $n$-fold monoid object in $\bB^\op$.
\end{remark}

\begin{remark}
  Some terminological caution is required: a lax $1$-fold monoid
  object in $\Cat$ is an ordinary monoid object; this is \emph{not}
  the same as a lax monoid object. A lax monoid object has
  inequational associativity (that is, associativity only up to a
  morphism rather than an isomorphism); the only structure that we are
  weakening in this way is (as usual) the interchange.
\end{remark}

We need these definitions in subsections \ref{ss:catoids},
\ref{ss:asb} and \ref{ss:double}.

We make the following conjecture:
\begin{conj}
Let $\cB$ be a monoidal bicategory. It should be possible to define a
monoidal bicategory $\Lax(\cB)$ of monoid objects, lax monoidal
$1$-cells, and monoidal $2$-cells in $\cB$, so that $\Lax(\Lax(\cB))$
is the monoidal bicategory of lax double monoid objects, and
$\Lax(\Lax(\Lax(\cB)))$ the monoidal bicategory of lax triple monoid
objects, and so on.

This would have the effect that $\Lax^n(\Cat)$ would be the monoidal
bicategory of $n$-fold monoidal categories, fully lax monoidal
functors and monoidal natural transformations.
\end{conj}

\section{Examples of $n$-fold monoidal categories}
\label{s:examples-nfold}

We discuss some examples of $n$-fold monoidal categories; further ones
can be found in \cite{AguMah}*{Section 6.4}.

\subsection{Operations on posets}
\label{ss:ops-posets}

Let $P$ be a poset, regarded as a category in the usual way. Giving
$P$ the structure of a monoidal category means imposing a monoid
structure on $P$, considered as a set, where the product is an
order-preserving function $P\times P\rightarrow P$. This data is known
as a \emph{partially ordered monoid}.

If there are multiple monoid structures on $P$, we can ask whether
they form an $n$-fold monoidal structure. In a poset category, all
diagrams commute, so only the maps $\chi$, $\zeta$, $\nu$ and $\iota$
are required to exist. These represent inequalities.

We give some examples of cases where they hold.
\begin{prop}
  If $(P,\odot,1)$ is a partially ordered commutative monoid, then $P$
  has the structure of an $n$-fold monoidal category with all monoidal
  structures given by $\odot$~and~$1$.
\end{prop}
\begin{proof}
  All four inequalities are in fact identities in any commutative
  monoid:
  \begin{align*}
    \chi:&&  (a\odot b)\odot(c\odot d)&\leq (a\odot c)\odot(b\odot d),\\
    \zeta:&& 1&\leq 1\odot 1,\\
    \nu:&&   1\odot 1&\leq 1,\\
    \iota:&& 1&\leq 1.\qedhere
    \end{align*}
\end{proof}

\begin{prop}
  \label{prop:po-monoid-join}
  Let $P$ be a partially ordered monoid which is a join-semilattice
  under the order. Then $P$ has the structure of a double monoidal
  category. Its first monoidal structure, $(\vee,\bot)$, arises from
  the semilattice, its second one, $(\odot,1)$, from the monoid on
  $P$.
\end{prop}
\begin{proof}
  The required inequalities hold in any join-semilattice with a monoid
  structure:
  \begin{align*}
    \chi:&&  (a\odot b)\vee(c\odot d)&\leq (a\vee c)\odot(b\vee d),\\
    \zeta:&& \bot&\leq \bot\odot\bot,\\
    \nu:&&   1\vee 1&\leq 1,\\
    \iota:&& \bot&\leq 1.
  \end{align*}
  Only the first inequality may require comment: this
  subdistributivity law is equivalent to order preservation.
\end{proof}

This proposition dualises:
\begin{prop}
  \label{prop:po-monoid-meet}
  Let $P$ be a partially ordered monoid which is also a
  meet-semilattice.  Then it has the structure of a double
  monoidal category with monoidal structures $(\odot,1)$ and
  $(\wedge,\top)$.\qed
\end{prop}

We consider similar structures with greater sophistication in
Section \ref{s:boolsemiring}.

There seems to be a wide range of similar structures. Addition and
multiplication on $\bN$ (with the standard ordering) do not extend to
a double monoidal structure: we have $0<1$ but $1+1>1$. However, if we
define a monoidal structure à la boolean rings by
$a\odot b = a + b + ab$ and with unit $0$, then
\begin{equation*}
(a\odot b)+(c\odot d)\leq(a+c)\odot(b+d).
\end{equation*}
Since the units agree, there is a double monoidal structure for $+$
and $\odot$.

The only extra requirements for a triple monoidal category, or an
$n$-fold monoidal category with $n>3$, are commutative diagrams.
These are automatically satisfied in a poset category and we can
extend these structures to $n$-fold monoidal categories in various
ways:
\begin{itemize}
\item Whenever one of the monoidal structures is commutative (as for
  instance join-semilattice or meet-semilattice structures), we can
  obtain an $(n+1)$-fold monoidal structure from an $n$-fold monoidal
  structure by repeating that monoidal structure.
\item We can chain together an $m$-fold with an $n$-fold monoidal
  structure whose last and first monoidal structures (respectively)
  agree to form an $(m+n-1)$-fold monoidal structure. For example, if
  $P$ is a partially ordered monoid and also a lattice (a
  \emph{lattice-ordered monoid}), then we can chain together the
  example of Propositions \ref{prop:po-monoid-join} and
  \ref{prop:po-monoid-meet} to get a triple monoidal category using
  $(\vee,\bot)$, $(\odot,1)$ and $(\wedge,\top)$ in turn.
\end{itemize}

There is a similar down-to-earth description of monoid objects in
poset categories with $n$-fold monoidal structures:
\begin{prop}
Let $P$ be a poset with an $(m+n)$-fold monoidal structure, with units
$1_1,\ldots,1_{m+n}$ and products $\odot_1,\ldots,\odot_{m+n}$. A
monoid object in $P$ of type $\text{lax}^m,\text{oplax}^n$ is an
object $x\in P$ with $1_m\leq x\leq 1_{m+1}$ and $x\odot_i x\leq x$
for $1\leq i\leq m$ and $x\leq x\odot_i x$ for $m+1\leq i\leq m+n$.
\end{prop}
\begin{proof}
This is a transcription of the structure, using again that all
diagrams in poset categories automatically commute.

For $1\leq i\leq m$, the monoid structures have units which give
$1_i\leq x$ and products which give $x\odot_ix\leq x$. However, since
we have $\iota$ giving $1_1\leq\cdots\leq 1_m$, we only need to impose
that $1_m\leq x$.

Dually, for $m+1\leq i\leq m+n$, the comonoid structures have counits
which give $x\leq 1_i$ and coproducts which give $x\leq x\odot_i
x$. Again, $\iota$ means that $x\leq 1_{m+1}$ is all that is needed.
\end{proof}

\subsection{Braided and symmetric monoidal categories}

Sufficiently structured monoidal categories give us $n$-fold
monoidal categories: we start with double monoidal categories (this is
essentially the material of \cite{JoyStr}*{Section 3}.)

\begin{prop}
\label{prop:braided-double}
Every braided monoidal category $\cC$, with product $\otimes$ and unit
$\I$, can be made into a double monoidal category, using $\otimes$ for
both products and $\I$ for both units.
\end{prop}
\begin{proof}
We take the interchange morphism
\[\chi:(A\otimes B)\otimes(C\otimes D)\rightarrow(A\otimes C)\otimes(B\otimes D)\]
from the braiding $B\otimes C\rightarrow C\otimes B$ (it is thus in
fact an isomorphism). The morphism $\iota$ is taken to be the
identity, and $\zeta$ and $\nu$ to be unit maps.

The only diagrams which do not simply move units around are D1 and
D11. D1 expresses the following identity of braids:
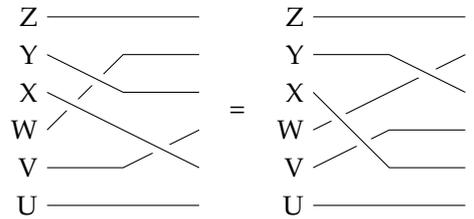
\begin{displaymath}
  \begin{tikzpicture}
    \begin{scope}[scale=0.5, shift={(-4,0)}]
      \node[left] at (0,0) {U};
      \node[left] at (0,1) {V};
      \node[left] at (0,2) {W};
      \node[left] at (0,3) {X};
      \node[left] at (0,4) {Y};
      \node[left] at (0,5) {Z};
      \draw (0,0) -- (4,0);
      \draw (0,5) -- (4,5);
      \draw (0,1) -- (2,1) -- (4,2);
      \draw (0,2) -- (2,4) -- (4,4);
      \draw[line width=2mm, white] (0,3) -- (4,1);
      \draw (0,3) -- (4,1);
      \draw[line width=2mm, white] (0,4) -- (2,3) -- (4,3);
      \draw (0,4) -- (2,3) -- (4,3);
    \end{scope}
    \node at (0.5,1.25) {=};
    \begin{scope}[scale=0.5, shift={(3,0)}]
      \node[left] at (0,0) {U};
      \node[left] at (0,1) {V};
      \node[left] at (0,2) {W};
      \node[left] at (0,3) {X};
      \node[left] at (0,4) {Y};
      \node[left] at (0,5) {Z};
      \draw (0,0) -- (4,0);
      \draw (0,5) -- (4,5);
      \draw (0,1) -- (2,2) -- (4,2);
      \draw (0,2) -- (4,4);
      \draw[line width=2mm, white] (0,3) -- (2,1) -- (4,1);
      \draw (0,3) -- (2,1) -- (4,1);
      \draw[line width=2mm, white] (0,4) -- (2,4) -- (4,3);
      \draw (0,4) -- (2,4) -- (4,3);
    \end{scope}
  \end{tikzpicture}.
\end{displaymath}
D11 expresses a similar identity.
\end{proof}

The following theorem is implicit in the work of Joyal and Street
\cite{JoyStr}:
\begin{prop}
\label{prop:symmetric-nfold}
Every symmetric monoidal category $\cC$ can be made into an $n$-fold
monoidal category for any $n$, using $\otimes$ for all the products,
and $\I$ for all the units.
\end{prop}

Joyal and Street also give a converse. They define degeneracy
conditions on double and triple monoidal categories so that they
determine braided or symmetric monoidal categories
(respectively). These conditions amount to the monoidal structures
being coherently isomorphic, and for the interchanges to be natural
isomorphisms. Although these are degenerate examples of $n$-fold
monoidal categories, they are nonetheless useful building blocks for
what follows.

We now turn to monoid objects in such categories, following
\cite{Porst}.
\begin{defn}
  Let $\cC$ be a braided monoidal category equipped with the double
  monoidal structure $\otimes_1=\otimes_2=\otimes$, as described
  above, and let $T$ be the $2$-fold monoidal functor type ``lax,
  oplax''. A monoid object in $\cC$ of type $T$ is then a known as a
  \emph{bimonoid object} in $\cC$.
\end{defn}

More explicitly, a bimonoid object in $\cC$ consists of an object $X$
equipped with the structure of a monoid object $m:X\otimes
X\rightarrow X$ and $u:\I\rightarrow X$, and with the structure of a
comonoid object $n:X\rightarrow X\otimes X$ and $v:X\rightarrow\I$,
which satisfies (in addition to the usual conditions for a monoid
object and a comonoid object), diagrams as follows:
\begin{center}
  \begin{tikzcd}
    (X\otimes X)\otimes(X\otimes X)\ar[dd, "\tau_{2,3}"']&
    X\otimes X\ar[d, "m"]\ar[l, "n\otimes n"']\\
    &X\ar[d, "n"]\\
    (X\otimes X)\otimes(X\otimes X)\ar[r, "m\otimes m"']&
    X\otimes X,\\
  \end{tikzcd}
  \hfil
  \begin{tikzcd}
    \I\ar[rr, "1"]\ar[dr, "u"']&&
    \I\\
    &X\ar[ur, "v"'],&
  \end{tikzcd}
\end{center}
\begin{center}
  \begin{tikzcd}
    \I\ar[r, "1"]\ar[d, "u"']&
    \I\otimes\I\ar[d, "u\otimes u"]\\
    X\ar[r, "n"']&X\otimes X,
  \end{tikzcd}
  \hfil
  \begin{tikzcd}
    \I\otimes\I\ar[r, "1"]&
      \I\\
      X\otimes X\ar[r, "m"']\ar[u, "v\otimes v"]&
      X\ar[u, "v"'].
  \end{tikzcd}
\end{center}
Here $\tau_{2,3}$ denotes the isomorphism $(w\otimes
x)\otimes(y\otimes z)\mapsto(w\otimes y)\otimes(x\otimes z)$ that
exchanges the second and third factors of $X$.

\begin{remark}
  If the underlying category is the category of vector spaces over a
  field $k$, then the result is precisely the usual concept of
  bialgebra \cite{Majid}.
\end{remark}

\begin{remark}
  If the underlying category is $\Rel$, monoid and comonoid objects
  coincide. The resulting structures are considered in
  \cite{CraDohStr1}.
\end{remark}

This construction extends to functors:
\begin{prop}
  \label{prop:braided-symmetric-functors}
  If $\cC$ and $\cD$ are braided monoidal categories, then a lax
  (respectively oplax, respectively strong) monoidal functor
  $\cC\rightarrow\cD$ gives a fully lax (respectively fully oplax,
  respectively fully strong) $2$-fold monoidal functor between $\cC$
  and $\cD$ thought of as double monoidal categories.

  If $\cC$ and $\cD$ are symmetric monoidal, then the same is true of
  the $n$-fold monoidal functor between $\cC$ and $\cD$ thought of as
  $n$-fold monoidal categories.
\end{prop}

We consider examples of this construction in the next subsection.

\subsection{Relational interchange monoids}
\label{ss:relational}

The results of the previous subsection allow us to consider lax double
monoid objects in the category $\Rel$. These extend the relational
monoids considered in Section \ref{s:examples-monoidal}, and have been
studied previously under the name \emph{relational interchange
monoids} in concurrency theory~\cite{CraDohStr1}.

Such a structure consists of a set $X$ and two ternary relations,
written as $P_{x,y}^z$ and $Q_{x,y}^z$, and corresponding to the
structure of $\odot$ and $\otimes$, respectively. There are also unit
subsets $U^x$ and $V^x$, such that $P$ and $U$, and $Q$ and $V$ both
form relational monoid structures, and the interchange laws are
satisfied. The latter can be expressed relationally as
\begin{align}
  \bigvee_{u,v}\left(Q_{wx}^u\wedge P_{uv}^t\wedge Q_{yz}^v\right)
    &\Longrightarrow
  \bigvee_{u,v}\left(P_{wy}^u\wedge Q_{uv}^t\wedge P_{xz}^v\right),
  \tag{$\chi$}\\
  U^z
    &\Longrightarrow
  \bigvee\left(U^x\wedge Q_{xy}^z\wedge U^y\right),
  \tag{$\zeta$}\\
  \bigvee\left(V^x\wedge P_{xy}^z\wedge V^y\right)
    &\Longrightarrow
  V^z,
  \tag{$\nu$}\\
  U^x
    &\Longrightarrow
  V^x.  \tag{$\iota$}
\end{align}

As in Section~\ref{s:examples-monoidal}, we can provide an important
language-theoretic example relevant to interleaving concurrency,
involving the shuffle product $\shuffle$ (considered in Example
\ref{ex:shuffle-relational}), and the concatenation $\cdot$ on the set
of words $M=A^*$. The former has previously been seen to be a
relational monoid structure; the latter is a relational monoid
structure arising from a genuine monoid.

It is routine to check that the two products satisfy the interchange
law, which here says that if $w$ is a shuffle of $u$ and $v$, and $w'$
is a shuffle of $u'$ and $v'$, then $w\cdot w'$ is a shuffle of
$u\cdot u'$ and $v\cdot v'$. The other interchange laws are trivial as
the unit for both monoid structures is the singleton set on the empty
word. This makes the set of words a lax double monoid in $\Rel$.

Later we will need lax $n$-fold monoids in $\Rel$ for $n\geq 3$; the
definition in Subsection \ref{ss:lax-monoid-objs} boils down to saying
that a \emph{lax $n$-fold relational monoid} is a set $X$ with $n$
relational monoid structures such that, for all $1\leq i<j\leq n$, the
$i$th and $j$th relational monoids form a lax double monoid.

\subsection{Categories with products and coproducts}
\label{ss:coprods-prods}

We give further examples to show that $n$-fold monoidal categories are
ubiquitous.  First, we show how they arise from categories with
coproducts and products. This is a useful building block: several
natural examples considered later have monoidal structures involving
disjoint unions and cartesian products.

Let $\cC$ be a category with all finite products and coproducts, and
let $m$ and $n$ be positive integers. We can make $\cC$ into an
$(m+n)$-fold monoidal category, by letting $\otimes_i=\sqcup$ and
$\I_i=0$ for $1\leq i\leq m$, and $\otimes_i=\times$ and $\I_i=1$ for
$m+1\leq i\leq m+n$, that is, using the cocartesian monoidal
structure $m$ times and the cartesian monoidal structure $n$ times.

For $1\leq i\leq j\leq m$, or for $m+1\leq i\leq j\leq m+n$, we
construct the structure ($\chi$, $\zeta$, $\nu$ and $\iota$) exactly
as in the previous subsection, using that the cocartesian and
cartesian monoidal structures are symmetric. Yet if $i\leq m$ and
$m+1\leq j$, we need to supply the structure. In this case, $\chi$ is
a morphism
\begin{equation*}
(W\times X)\sqcup(Y\times Z)\rightarrow(W\sqcup Y)\times(X\sqcup Z).
\end{equation*}
By the universal properties of coproducts and products, it can be
specified using four maps, one from each summand on the left to each
factor on the right. Each is defined by a projection followed by an
inclusion:
\begin{center}
  \begin{tabular}{cc}
    $W\times X\rightarrow W\rightarrow W\sqcup Y$,&
    $W\times X\rightarrow X\rightarrow X\sqcup Z$,\\
    $Y\times Z\rightarrow Y\rightarrow W\sqcup Y$,&
    $Y\times Z\rightarrow Z\rightarrow X\sqcup Z$.
  \end{tabular}
\end{center}
The maps $\zeta: 0\rightarrow0\times0$, $\nu: 1\sqcup1\rightarrow1$
and $\iota:0\rightarrow1$ are uniquely determined by being either from
initial objects, or to terminal objects. In the special case where the
category is a poset, we recover \cite{AguMah}*{Example 6.20}.

Of the twelve diagrams defining a double monoidal structure, D8, D10
and D12 commute as they compare maps out of an initial object, while
D4, D5 and D6 commute as they compare maps into a terminal
object. Checking the other diagrams requires more work, but is still
straightforward by reducing to summands (of the sources of the
diagrams) and factors (of the targets) using the universal properties.

Triple interchange diagrams require verification when
$1\leq i<j\leq m$ and $m+1\leq k\leq n$, and when $1\leq i\leq m$ and
$m+1\leq j<k\leq n$. These two cases are dual (since coproducts and
products are formally dual). We consider the latter. All but two of
the diagrams automatically commute since they either directly compare
maps out of an initial object or maps into a terminal object. Of the
remaining two, diagram T4 also compares maps into a terminal object
$1\times 1$:
  \begin{center}
    \begin{tikzcd} 
      (1\times1)\sqcup(1\times1)
      \ar[rr, "\chi_{ik}"]&&
      (1\sqcup1)\times(1\sqcup1)
      \ar[d, "\nu_{ij}\times\nu_{ij}"]\\
      1\sqcup1
      \ar[u, "\zeta_{jk}\sqcup\zeta_{jk}"]
      \ar[r, "\nu_{ij}"']&
      1
      \ar[r, "\zeta_{jk}"']&
      1\times1.
    \end{tikzcd}
  \end{center}
The final diagram, T1, is large, but can  be broken
down easily by summands and factors, and checked as before:
  \begin{center}
    \begin{tikzcd} 
      \begin{tabular}{C}
        ((S\times T)\times(U\times V))\\
        \sqcup\\
        ((W\times X)\times(Y\times Z))
      \end{tabular}
      \ar[r, "\chi_{jk}\sqcup\chi_{jk}"]
      \ar[d, "\chi_{ij}"']&
      \begin{tabular}{C}
        ((S\times U)\times(T\times V))\\
        \sqcup\\
        ((W\times Y)\times(X\times Z))
      \end{tabular}
      \ar[d, "\chi_{ik}"]\\
      \begin{tabular}{C}
        ((S\times T)\sqcup(W\times X))\\
        \times\\
        ((U\times V)\sqcup(Y\times Z))
      \end{tabular}
      \ar[d, "\chi_{ik}\times\chi_{ik}"']&
      \begin{tabular}{C}
        ((S\times U)\sqcup(W\times Y))\\
        \times\\
        ((T\times V)\sqcup(X\times Z))
      \end{tabular}
      \ar[d, "\chi_{ij}\times\chi_{ij}"]\\
      \begin{tabular}{C}
        ((S\sqcup W)\times(T\sqcup X))\\
        \times\\
        ((U\sqcup Y)\times(V\sqcup Z))
      \end{tabular}
      \ar[r, "\chi_{jk}"']
      &
      \begin{tabular}{C}
        ((S\sqcup W)\times(U\sqcup Y))\\
        \times\\
        ((T\sqcup X)\times(V\sqcup Z)).
      \end{tabular}
    \end{tikzcd}
  \end{center}

\begin{remark}
  This example extends in a straightforward
  fashion to give an $n$-fold monoidal category where all $n$ monoidal
  structures are different.

  Consider a category $\cC$ with finite coproducts and finite
  products. We can define $n$ different monoidal structures,
  $(\otimes_i,\I_i)$ for $1\leq i\leq n$ on $\cC^{n-1}$ by taking
  \begin{align*}
    (X_1,\ldots,X_{n-1})\otimes_i(Y_1,\ldots,Y_{n-1}) = (&X_1\times Y_1,\ldots,X_{i-1}\times Y_{i-1},\\
    &X_i\sqcup Y_i,\ldots,X_{n-1}\sqcup Y_{n-1})\\
    \I_i = (&1_1,\ldots,1_{i-1},0_i,\ldots,0_{n-1}).
  \end{align*}
  This can be completed to an $n$-fold monoidal category by adding the
  rest of the structure componentwise: if we do so, each component
  looks like the example above, but the $i$th component has $n-i-1$
  uses of the cartesian monoidal structure and $i$ uses of the
  cocartesian one.
\end{remark}

\subsection{Sequences of morphisms and communication protocols}

A more elaborate version of the preceding example is a construction
relevant to modelling client/server protocols (which can be modelled
by dependent lenses~\cite{VidCap}). A simple model of a stateless
client/server protocol is given by a morphism of sets $p:A\rightarrow
Q$: the client may send a request $q\in Q$, to which the server will
send a response $a\in A_q=p^{-1}(q)$. We can regard $A$ as a dependent
response type with a parameter drawn from $Q$.

The arrow category of $\Set$ then gives a natural category of
client/server protocols: a morphism from $p_1:A_1\rightarrow Q_1$ to
$p_2:A_2\rightarrow Q_2$ consists of a map on the type of requests
$f:Q_1\rightarrow Q_2$ and a map of response types
$g_q:p_1^{-1}(q)\rightarrow p_2^{-1}(f(q))$ for each possible request
$q\in Q_1$.
\begin{displaymath}
  \begin{tikzcd}
    A_1\ar[r, "g"]\ar[d, "p_1"'] & A_2\ar[d, "p_2"]\\
    Q_1\ar[r, "f"'] & Q_2
  \end{tikzcd}
\end{displaymath}

This suggests studying monoidal structures on the arrow category of
$\Set$. The constructions of the previous subsection with products and
coproducts can be combined to give three monoidal structures, each of
which corresponds to a natural combination of protocols.
\begin{description}
\item[Client-side choice:] the pointwise disjoint union $A_1\sqcup
  A_2\rightarrow Q_1\sqcup Q_2$ represents the client choosing between
  one of two request types, and the server responding accordingly.
\item[Server-side choice:] the protocol $(A_1\times
  Q_2)\sqcup(Q_1\times A_2)$, with the natural map to $Q_1\times Q_2$,
  has preimages
  \[\left((A_1\times Q_2)\sqcup(Q_1\times A_2)\right)_{(q_1,q_2)} = {(A_1)}_{q_1}\sqcup {(A_2)}_{q_2}.\]
  It represents the client making requests of both types, and the
  server choosing one of the two and responding to it.
\item[Simultaneous:] the pointwise cartesian product
  $A_1\times A_2\rightarrow Q_1\times Q_2$ of protocols
  $A_1\rightarrow Q_1$ and $A_2\rightarrow Q_2$ has preimages
  ${(A_1\times A_2)}_{(q_1,q_2)} = {(A_1)}_{q_1} \times
  {(A_2)}_{q_2}$. This represents the client making requests of both
  types simultaneously and the server responding to both
  simultaneously.
\end{description}
These three are all symmetric monoidal, with units $0\rightarrow 0$,
$0\rightarrow 1$ and $1\rightarrow 1$ respectively.

It is not difficult to show that these form a triple monoidal category
(indeed, by repetition of structure, an $(i+j+k)$-fold monoidal
category for any $i$, $j$, $k$). If $j=0$, the $(i+k)$-fold monoidal
structure is the one obtained from coproducts and products in
subsection \ref{ss:coprods-prods}, but on the arrow category of
$\Set$; the novel monoidal structure is the middle one, ``server-side
choice''.


This generalises to $n$-stage communication protocols, which can be
modelled by a chain
$X_n\rightarrow\cdots\rightarrow X_2\rightarrow X_1$ of morphisms. At
the $i$th stage a message in $X_i$ is chosen; at the next
stage the next message must be in its preimage in $X_{i+1}$. This has
$n+1$ distinct symmetric monoidal structures
$\otimes_0,\ldots,\otimes_n$, defined by
\begin{displaymath}
  (X_*\otimes_r Y_*)_i =
  \begin{cases}
    X_i\times Y_i, & \text{if $i\leq r$;}\\
    (X_i\times Y_r)\sqcup(X_r\times Y_i), & \text{if $i>r$.}
  \end{cases}
\end{displaymath}
The symmetric monoidal structure $\otimes_r$ represents ``making a
choice after the $r$th step'' (so, if $r=n$, no choice is made, and it
is the pointwise cartesian product; if $r=0$ the choice is made right
from the beginning, and it is the pointwise disjoint union); the unit is
\begin{displaymath}
  (\I_r)_i =
  \begin{cases}
    1, & \text{if $i\leq r$;}\\
    0, & \text{if $i>r$.}
  \end{cases}
\end{displaymath}
By repetition of structure, this can be assembled to an
$(i_0+\cdots+i_n)$-fold monoidal structure for any $i_0,\ldots,i_n$.

\subsection{Monoidal structures commuting with (co)products}
\label{ss:bicontinuous}

We now consider another common source of examples, covering, for
instance,
\begin{itemize}
\item the category $\Rel$ of sets and relations, together with the
  (setwise) cartesian product monoidal structure,
\item the category $\Ab$ of abelian groups and homomorphisms, together
  with the tensor product monoidal structure, and
\item the category $\Vect$ of vector spaces and linear maps, together
  with the tensor product monoidal structure.
\end{itemize}

The input to the construction is a monoidal category $\cC$ with finite
coproducts, where the monoidal structure commutes with coproducts in
each variable separately, that is, where the canonical morphisms
\begin{align*}
  0&\longrightarrow 0\otimes X,\\
  0&\longrightarrow X\otimes 0,\\
  X\otimes Y\sqcup X\otimes Z&\longrightarrow X\otimes(Y\sqcup Z),\\
  X\otimes Z\sqcup Y\otimes Z&\longrightarrow(X\sqcup Y)\otimes Z
\end{align*}
are isomorphisms for all objects $X$, $Y$ and $Z$.

Our basic result on these examples is as follows.
\begin{prop}
  \label{prop:continuous-monoidal}
  Let $\cC$ be a category with finite coproducts and a monoidal
  structure $(\otimes,\I)$ that commutes with coproducts in each
  variable separately. Then $\cC$ has an $n$-fold monoidal structure
  for any $n\geq 0$, with $(\otimes_i,\I_i) = (\sqcup, 0)$ for
  $1\leq i< n$ and $(\otimes_n,\I_n) = (\otimes,\I)$.
\end{prop}
\begin{proof}
Theorem \ref{prop:symmetric-nfold} applied to the cocartesian monoidal
structure on $\cC$ yields an $(n-1)$-fold monoidal structure; we
only need to extend it to an $n$-fold monoidal structure using
$(\otimes_n,\I_n) = (\otimes,\I)$.

The structure maps are defined by taking $\chi_{in}$ to be the
inclusion map
\begin{displaymath}
  W\otimes X\sqcup Y\otimes Z
    \longrightarrow
  W\otimes X\sqcup W\otimes Z\sqcup Y\otimes X\sqcup Y\otimes Z
    \stackrel{\sim}{\longrightarrow}
  (W\sqcup Y)\otimes(X\sqcup Z),
\end{displaymath}
$\zeta_{in}$ to be the initial map $0\rightarrow 0\otimes 0$ (which is
an isomorphism), $\nu_{in}$ to be the fold map
$\I\sqcup\I\rightarrow \I$, and $\iota_{in}$ to be the initial map
$0\rightarrow\I$.

Of the double interchange diagrams, D8, D10 and D12 are automatic
since they compare maps out of an initial object. The remaining nine
are maps out of coproducts which can readily be checked on each
summand.

We must also check the triple interchange diagrams for $i<j<n$. Again,
T3--8 compare maps out of an initial object, and the remaining two can
be checked summand-by-summand.
\end{proof}

\begin{remark}
  \label{rmk:cocontinuous}
  The dual situation is that of a category $\cC$ with finite
  \emph{products} and a monoidal structure $(\otimes,\I)$ that
  commutes with products in each variable separately, which carries an
  $n$-fold monoidal structure where the last $(n-1)$ monoidal
  structures are the cartesian one, and the first is $(\otimes,\I)$.
\end{remark}

We can vary this construction, according to the philosophy that
repeating a monoidal structure once constitutes a braiding.
\begin{prop}
  \label{prop:continuous-braided}
  In the situation of Proposition \ref{prop:continuous-monoidal}, if
  $(\cC,\otimes,\I)$ is braided monoidal, then $\cC$ carries, for any
  $n\geq 0$, an $(n+2)$-fold monoidal structure where the first $n$
  monoidal structures are cocartesian, and the last two are
  $(\otimes,\I)$.
\end{prop}
\begin{proof}
Most of the structure is given by Proposition
\ref{prop:braided-double}, which gives a double monoidal structure for
the last two monoidal structures, and \ref{prop:continuous-monoidal},
which gives all the other parts of the structure not involving both
the last two monoidal structures. Only their compatibility, specified
by the triple interchange diagrams where $i<n+1<n+2$, need to be
checked.

Of these, T3, T6, T7 and T8 compare maps out of an initial object; T2,
T4 and T5 all compare two maps both of which are essentially the fold
map $\I\sqcup\I\rightarrow\I$. The diagram T1, however, is deserving
of comment.
  \begin{displaymath}
    \begin{tikzcd} 
      \begin{tabular}{C}
        ((S\otimes T)\otimes(U\otimes V))\\
        \sqcup \\
        ((W\otimes X)\otimes(Y\otimes Z))
      \end{tabular}
      \ar[r, "\chi\sqcup \chi"]
      \ar[d, "\chi"']&
      \begin{tabular}{C}
        ((S\otimes U)\otimes (T\otimes V))\\
        \sqcup \\
        ((W\otimes Y)\otimes (X\otimes Z))
      \end{tabular}
      \ar[d, "\chi"]\\
      \begin{tabular}{C}
        ((S\otimes T)\sqcup (W\otimes X))\\
        \otimes \\
        ((U\otimes V)\sqcup (Y\otimes Z))
      \end{tabular}
      \ar[d, "\chi\otimes\chi"']&
      \begin{tabular}{C}
        ((S\otimes U)\sqcup (W\otimes Y))\\
        \otimes \\
        ((T\otimes V)\sqcup (X\otimes Z))
      \end{tabular}
      \ar[d, "\chi\otimes\chi"]\\
      \begin{tabular}{C}
        ((S\sqcup W)\otimes (T\sqcup X))\\
        \otimes \\
        ((U\sqcup Y)\otimes (V\sqcup Z))
      \end{tabular}
      \ar[r, "\chi"']
      &
      \begin{tabular}{C}
        ((S\sqcup W)\otimes (U\sqcup Y))\\
        \otimes\\
        ((T\sqcup X)\otimes (V\sqcup Z)).
      \end{tabular}
    \end{tikzcd}
  \end{displaymath}
  Reducing the domain in the top left to the two summands, both sides
  can be checked to be a map $\chi$ followed by a summand inclusion.
\end{proof}

We can similarly vary the statement according to the philosophy that
repeating the same monoidal structure more than once constitutes a
braiding.
\begin{prop}
\label{prop:continuous-symmetric}
In the situation of Proposition \ref{prop:continuous-monoidal}, where
$(\cC,\otimes,\I)$ is symmetric monoidal, for any $n\geq 0$, the
category $\cC$ carries an $(m+n)$-fold monoidal structure where the
first $m$ monoidal structures are cocartesian, and the last $n$ are
$(\otimes,\I)$.
\end{prop}
\begin{proof}
  This requires no extra work: all the interchange diagrams are
  covered jointly by Propositions \ref{prop:symmetric-nfold} and
  \ref{prop:continuous-braided}.
\end{proof}

Lastly, we can go further with monoidal categories with biproducts
(that is, where finite coproducts and products exist and agree. We use
the notation $0$ for the object which is both initial and terminal,
and $X\oplus Y$ for the object which is both the product and coproduct
of $X$ and $Y$).
\begin{prop}
  \label{prop:bicontinuous-monoidal}
Let $(\cC,\otimes,\I)$ be a monoidal category with finite biproducts,
where the monoidal structure commutes with biproducts in each variable
separately. Then, for any $m,n\geq 0$, the category $\cC$ carries a
$(m+1+n)$-fold monoidal structure, where the first $m$ and last $n$
structures are cartesian, and the middle one is $(\otimes,\I)$.
\end{prop}
\begin{proof}
The only interchange diagrams not covered by the propositions above
are the triple interchange diagrams for $i<m+1<k$. Of these, all but
T1 and T4 compare maps out of or into the zero object (which is both
initial and terminal).
  \begin{displaymath}
    \begin{tikzcd} 
      \begin{tabular}{C}
        ((S\oplus T)\otimes (U\oplus V))\\
        \oplus \\
        ((W\oplus X)\otimes (Y\oplus Z))
      \end{tabular}
      \ar[r, "\chi\oplus\chi"]
      \ar[d, "\chi"']&
      \begin{tabular}{C}
        ((S\otimes U)\oplus (T\otimes V))\\
        \oplus \\
        ((W\otimes Y)\oplus (X\otimes Z))
      \end{tabular}
      \ar[d, "\chi"]\\
      \begin{tabular}{C}
        ((S\oplus T)\oplus (W\oplus X))\\
        \otimes \\
        ((U\oplus V)\oplus (Y\oplus Z))
      \end{tabular}
      \ar[d, "\chi\otimes\chi"']&
      \begin{tabular}{C}
        ((S\otimes U)\oplus (W\otimes Y))\\
        \oplus \\
        ((T\otimes V)\oplus (X\otimes Z))
      \end{tabular}
      \ar[d, "\chi\oplus\chi"]\\
      \begin{tabular}{C}
        ((S\oplus W)\oplus (T\oplus X))\\
        \otimes \\
        ((U\oplus Y)\oplus (V\oplus Z))
      \end{tabular}
      \ar[r, "\chi"']
      &
      \begin{tabular}{C}
        ((S\oplus W)\otimes (U\oplus Y))\\
        \oplus \\
        ((T\oplus X)\otimes (V\oplus Z)),
      \end{tabular}
    \end{tikzcd}
  \end{displaymath}
  \begin{displaymath}
    \begin{tikzcd} 
      (\I\oplus \I)\oplus (\I\oplus \I)
      \ar[rr, "\chi"]&&
      (\I\oplus \I)\oplus (\I\oplus \I)
      \ar[d, "\nu\oplus\nu"]\\
      \I\oplus \I
      \ar[u, "\zeta\oplus\zeta"]
      \ar[r, "\nu"']&
      \I
      \ar[r, "\zeta"']&
      \I\oplus \I.
    \end{tikzcd}
  \end{displaymath}
  The first diagram, T1, compares maps between two objects which are
  both biproducts of eight products of two objects. These maps can
  hence be resolved into sixty-four components. Both maps are the
  identity on the shared components $S\otimes U$, $T\otimes V$,
  $W\otimes Y$ and $X\otimes Z$, and zero between all other pairs of
  components. Hence they agree.

  The second diagram, T4, can be resolved into four components, and is
  the identity on each.
\end{proof}

This can be strengthened, as one might expect from the above.
\begin{prop}
  \label{prop:bicontinuous-braided-symmetric}
  In the situation of the proposition above, if $(\cC,\otimes,\I)$ is
  braided, then (for any $m,n\geq 0$) there is a $(m+2+n)$-fold
  monoidal structure of the same sort. Indeed, if $(\cC,\otimes,\I)$
  is symmetric, then (for any $m,k,n\geq 0$) there is a $(m+k+n)$-fold
  monoidal structure of the same sort.
\end{prop}
\begin{proof}
  In both cases, all diagrams required are covered by the propositions
  above.
\end{proof}

\subsection{Posets and related constructions}
\label{ss:posets-et-cetera}

The category $\cP$ of posets and order-preserving maps can be equipped
with two monoidal structures that act as an elementwise disjoint
union
\[\Ob(A\sqcup B)=\Ob(A\odot B)=\Ob A\sqcup\Ob B,\]
but that act differently on the orderings:
\begin{itemize}
\item in the \emph{disjoint union} monoidal structure, $A\sqcup B$ is
  such that every element in $A$ is incomparable with every element in
  $B$ (but the internal order on $A$ and $B$ is preserved).
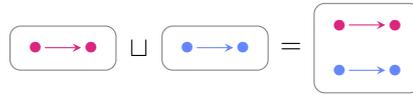
\begin{displaymath}
\begin{tikzpicture}
\node[draw, black!50, inner sep=2pt, rounded corners] at (-3,0) {
  \begin{tikzcd}[row sep=tiny, column sep=tiny, arrow style=tikz, every arrow/.append style={shorten=-4pt}]
    \bulletred\ar[r,cbred]&\bulletred
  \end{tikzcd}};
\node at (-2,0) {$\sqcup$};
\node[draw, black!50, inner sep=2pt, rounded corners] at (-1,0) {
  \begin{tikzcd}[row sep=tiny, column sep=tiny, arrow style=tikz, every arrow/.append style={shorten=-4pt}]
    \bulletblue\ar[r,cbblue]&\bulletblue
  \end{tikzcd}};
\node at (0,0) {$=$};
\node[draw, black!50, inner sep=2pt, rounded corners] at (1,0) {
  \begin{tikzcd}[row sep=tiny, column sep=tiny, arrow style=tikz, every arrow/.append style={shorten=-4pt}]
    \bulletred\ar[r,cbred]&\bulletred\\
    \bulletblue\ar[r,cbblue]&\bulletblue
  \end{tikzcd}};
\end{tikzpicture}
\end{displaymath}
\item in the \emph{join} monoidal structure, $A\odot B$ is such that
  every element in $A$ is less than every element in $B$ (and the
  internal order on $A$ and $B$ is also preserved).
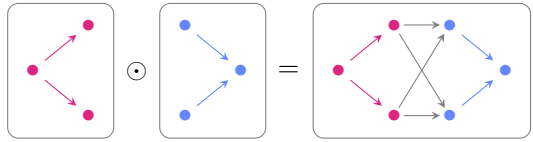
\begin{displaymath}
\begin{tikzpicture}
\node[draw, black!50, inner sep=2pt, rounded corners] at (-3,0) {
  \begin{tikzcd}[row sep=tiny, column sep=tiny, arrow style=tikz, every arrow/.append style={shorten=-4pt}]
    & \bulletred\\
    \bulletred\ar[ur, cbred]\ar[dr, cbred] &\\
    & \bulletred
  \end{tikzcd}};
\node at (-2,0) {$\odot$};
\node[draw, black!50, inner sep=2pt, rounded corners] at (-1,0) {
  \begin{tikzcd}[row sep=tiny, column sep=tiny, arrow style=tikz, every arrow/.append style={shorten=-4pt}]
    \bulletblue\ar[dr, cbblue] &\\
    & \bulletblue\\
    \bulletblue\ar[ur, cbblue] &
  \end{tikzcd}};
\node at (0,0) {$=$};
\node[draw, black!50, inner sep=2pt, rounded corners] at (1.75,0) {
  \begin{tikzcd}[row sep=tiny, column sep=tiny, arrow style=tikz, every arrow/.append style={shorten=-4pt}]
    & \bulletred\ar[r]\ar[ddr] & \bulletblue\ar[dr, cbblue] &\\
    \bulletred\ar[ur, cbred]\ar[dr, cbred] &&& \bulletblue\\
    & \bulletred\ar[r]\ar[uur] & \bulletblue\ar[ur, cbblue] &
  \end{tikzcd}};
\end{tikzpicture}
\end{displaymath}
\end{itemize}
This construction is again relevant to concurrency. In partial order
semantics of concurrency~\cite{Vogler}, $\odot$ models a ``serial'' or
``sequential'' composition of programs or systems and $\sqcup$ a
``parallel'' composition. Indeed, when we run two tasks in series, we
demand that every event in the former happens before every event in
the latter. When we run two tasks in parallel, however, we only assume
that events occur in their proper order within each task.

Both $\sqcup$ and $\odot$ have the empty poset as their unit, and the
existence of the interchange morphism, which in this case is an
inclusion of posets bijective on objects (aka \emph{subsumption}), is
clear:
\begin{equation*}
  (A\odot B)\sqcup(C\odot D)\rightarrow(A\sqcup C)\odot (B\sqcup D).
\end{equation*}
This map is usually not an isomorphism: in the right-hand side,
elements of $B$ precede those of $C$, but in the left-hand side
they are incomparable. Similarly, in the right-hand side, elements of
$A$ precede elements of $D$, whereas in the left-hand side they are
incomparable. In this sense, the left-hand side poset is ``more
parallel'' and the right-hand side poset ``more serial''.

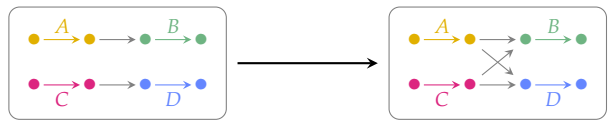
\begin{displaymath}
\begin{tikzpicture}
\node[draw, black!50, inner sep=2pt, rounded corners] (ldiag) at (-2.5,0) {
  \begin{tikzcd}[row sep=tiny, column sep=tiny, arrow style=tikz, every arrow/.append style={shorten=-4pt}]
    \bulletorange\ar[r,cborange,"A"]&\bulletorange\ar[r]&\bulletteal\ar[r,cbteal,"B"]&\bulletteal\\
    \bulletred\ar[r,cbred,"C"']&\bulletred\ar[r]&\bulletblue\ar[r,cbblue,"D"']&\bulletblue\\
  \end{tikzcd}};
\node[draw, black!50, inner sep=2pt, rounded corners] (rdiag) at (2.5,0) {
  \begin{tikzcd}[row sep=tiny, column sep=tiny, arrow style=tikz, every arrow/.append style={shorten=-4pt}]
    \bulletorange\ar[r,cborange,"A"]&\bulletorange\ar[r]\ar[dr]&\bulletteal\ar[r,cbteal,"B"]&\bulletteal\\
    \bulletred\ar[r,cbred,"C"']&\bulletred\ar[r]\ar[ur]&\bulletblue\ar[r,cbblue,"D"']&\bulletblue\\
  \end{tikzcd}};
  \draw[thick, ->, shorten >=4pt, shorten <=4pt] (ldiag)--(rdiag);
\end{tikzpicture}
\end{displaymath}

It is straightforward to check that both these constructions are
monoidal. In fact, for any $m$ we can make an $(m+1)$-fold monoidal
category with $\otimes_i=\sqcup$ for $1\leq i\leq m$ and
$\otimes_{m+1}=\odot$.

We can use products of posets as well, extending this example with an
extra layer of monoidal structure.
\begin{prop}
  \label{prop:posetcat}
  For any natural numbers $m$, $n$, there is an $(m+1+n)$-fold
  monoidal category with $\otimes_i=\sqcup$ for $1\leq i\leq m$, and
  $\otimes_{m+1}=\odot$, and $\otimes_i=\times$ for $m+1+1\leq i\leq
  m+1+n$.
\end{prop}
\begin{proof}
  The structure on the first $m+1$ products, relating $\sqcup$ and
  $\odot$, is that just constructed.

  The double monoidal structure relating $\sqcup$ and $\times$ is that
  of Subsection \ref{ss:coprods-prods}. The double monoidal structure
  relating $\odot$ and $\times$ is similar: there is an inclusion
  \[\chi:(A\times B)\odot(C\times D)\longrightarrow(A\odot C)\times(B\odot D),\]
  and we use the isomorphism $\zeta: 0\rightarrow 0\times 0$, and the
  terminal maps $\nu: 1\odot 1\rightarrow 1$ and $\iota:0\rightarrow
  1$.

  Checking the double interchange diagrams for this structure and the
  new triple interchange diagrams is tedious but straightforward
  (much like the checks in Subsection \ref{ss:coprods-prods}).
\end{proof}

These constructions above are quite specific, and generalise or adapt
in a number of ways without essential change:
\begin{itemize}
\item we can require that posets be finite;
\item we can drop antisymmetry  and hence consider preorders;
\item we can also drop transitivity and work with directed or directed
  acyclic graphs;
\item we can drop transitivity and add symmetry, thus working with
  graphs;
\item we can work with categories instead (where, as usual, a
  morphism generalises an arrow in a Hasse diagram for a preorder),
  using a categorical join \cite{FriGol} for $\odot$;
\item we can work with simplicial sets and a simplicial join
  \cite{FriGol} operation for $\odot$;
\item we can attach extra structure, such as labellings for vertices
  (as used in the \emph{pomsets} or \emph{partial words} of
  \cites{Wink, Grabowski} and the labelled transition systems of
  concurrency theory) or for edges (as used in automata).
\end{itemize}

In all these settings, the induced maps of Subsection
\ref{ss:induced-maps} are interesting in relating the sequential and
parallel monoidal structures: the units are uniquely isomorphic, and
the map $\varphi(1):A\sqcup B\rightarrow A\odot B$ is the natural
inclusion of posets.

In the theory of pomsets, cited above, it is normal to pass to
isomorphism classes of vertex-labelled posets. In general, just as a
monoidal category has a monoid of isomorphism classes, an $n$-fold
monoidal category has an $n$-fold monoid of isomorphism classes. Thus
we have the following fact.
\begin{prop}
\label{prop:pomsets}
Pomsets form an $(m+1+n)$-fold monoid, for any $m$ and $n$, where the
first $m$ monoidal structures are induced by $\sqcup$, the next by
$\odot$, and the last $n$ by $\times$.
\end{prop}

\subsection{Functorial constructions on posets}

Many natural constructions on posets are functorial, and several
phenomena of interest in concurrency can be expressed as functors from
a category of posets. In fact, the category of posets is monoidal
closed: for any posets $P$ and $Q$, the homset $\Poset(P,Q)$ naturally
has the structure of a poset, with the pointwise order (that is, where
$f\leq g$ if $f(x)\leq g(x)$ for all $x\in P$) as internal hom, and
$\Poset(P,\Poset(Q,R))\isom\Poset(P\times Q,R)$.

For any phenomenon that can be expressed using this closed structure,
one can ask how it interacts with the various monoidal structures,
described above, on the category of posets. We give some interesting
and useful examples.

\begin{example}
Consider the set of intervals $x\leq y$ in a poset. Writing $2$ for
the poset $0<1$, the ``set of intervals'' functor is, in fact, the
corepresentable functor $\Poset(2,-):\Poset\rightarrow\Poset$.

Since this functor is corepresentable, it is strong monoidal on the
cartesian monoidal structure: we have isomorphisms $1\isom\Poset(2,1)
$ and
\[\Poset(2,P)\times\Poset(2,Q)\isom\Poset(2,P\times Q).\]

Restricting to ``strict intervals'' in $P$ (that is, pairs satisfying
$x<y$ rather than $x\leq y$) has weaker categorical properties: this
is only functorial on the category $\Poset^\mono$ of posets and
strictly monotone maps (since if we have $f:P\rightarrow Q$ and $x<y$,
then we must also have $f(x)<f(y)$ for this strict interval in $P$ to
be sent to a strict interval in $Q$). The category of posets and
strictly monotone maps does not have categorical products (nor even a
terminal object), but it still has a monoidal structure given by
cartesian products. The strict interval functor is neither lax nor
oplax for the cartesian monoidal structure: we have natural maps
$\Poset^\mono(2,1)=\emptyset\rightarrow 1$ (compatible with an oplax
monoidal structure), but
\[\Poset^\mono(2,P)\times\Poset^\mono(2,Q)\longrightarrow\Poset^\mono(2,P\times Q)\]
(compatible with a lax monoidal structure).

Things work better for other monoidal structures: both the interval
and strict interval functors give strong monoidal maps with respect to
the disjoint union structure. Also, both are lax monoidal with respect
to the join monoidal structure $\odot$ described at the beginning of
this section: we have natural maps
\[\Poset(2,P)\odot\Poset(2,Q)\rightarrow \Poset(2,P\odot Q)\]
and $\Poset(2,\emptyset)\rightarrow\emptyset$, and the same also in
$\Poset^\mono$.

In fact, the interval functor can be extended to:
\begin{itemize}
\item a monoidal functor of type ``$\text{lax}^{m+1},\text{strong}^n$''
  from the $(m+1+n)$-fold monoidal category described in Proposition
  \ref{prop:posetcat}, or
\item a monoidal functor of type ``$\text{lax}^{m+1},\text{strong}^n$'',
  from the $(m+1+n)$-fold monoidal category which uses $\sqcup$ for
  the first $m+1$ monoidal products, and $\times$ for the last $n$, to
  the category in Proposition \ref{prop:posetcat}.
\end{itemize}

Similar comments apply to other diagram categories, and their strict
versions (that is, $\Poset(D,-)$ and $\Poset^\mono(D,-)$ for other
choices of $D$).
\end{example}

\begin{example}
A \emph{downset} of a poset $P$ is a subset $S\subset P$ such that if
$y\in S$ and $x\leq y$ then $x\in S$. Downsets can be studied
systematically as the functor category $\bD(P)=\Poset(-,2)$ (with the
downset being the preimage of $0$ in $P$). The functor $\bD$ is then a
functor $\Poset^\op\rightarrow\Poset$: functoriality is contravariant
in the sense that of $S\subset Q$ is a downset and $f:P\rightarrow Q$
a homomorphism, then $f^{-1}(S)$ is a downset of $P$. One can also
study \emph{upsets}, but this is equivalent: the complement of a
downset is an upset and vice versa.

As a representable functor, $\bD$ takes coproducts to products and
is therefore a strong monoidal functor from $\sqcup$ to
$\times$. There are also a number of lax monoidal structures on
$\bD$: from $\odot$ to $\odot$, from $\sqcup$ to $\odot$, and from
$\sqcup$ to $\sqcup$.

Again, this functor can be assembled to a functor of $n$-fold monoidal
categories in a number of ways compatible with the structures
above. For example, here is one coherence diagram expressing that $\bD$ is
a ``lax, strong'' monoidal functor from $(\sqcup,\sqcup)$ to
$(\odot,\times)$:
  \begin{center}
    \begin{tikzcd}
      (\bD W\times \bD X)\odot(\bD Y\times \bD Z)\ar[r, "\chi"]\ar[d, "\wr"']&
      (\bD W\odot \bD Y)\times(\bD X\odot \bD Z)\ar[d, "\mu\otimes\mu"]\\
      \bD (W\sqcup X)\odot \bD(Y\sqcup Z)\ar[d, "\mu"']&
      \bD (W\sqcup Y)\times \bD(X\sqcup Z)\ar[d, "\wr"]\\
      \bD ((W\sqcup X)\sqcup(Y\sqcup Z))\ar[r, "\bD(\chi)"']&
      \bD ((W\sqcup Y)\sqcup(X\sqcup Z)).
    \end{tikzcd}
  \end{center}
This can be shown to commute by checking components. An element
$(A,B)$ of $\bD W\times \bD X$ is sent to the corresponding subset of
$W\sqcup X\sqcup Y\sqcup Z$ containing none of $Y$ and $Z$; an element
$(A,B)$ of $\bD Y\times \bD Z$ is sent to the corresponding subset of
$W\sqcup X\sqcup Y\sqcup Z$ containing all of $W$ and $X$.
\end{example}

\begin{example}
The set of linearisations $\Lin(P)$ of a poset $P$ (that is, the set
of poset maps to a linear ordering which are bijective on objects) can
be defined to be a functor from the category $\Poset^{\inj,\op}$ (the
opposite of the category of posets and injective homomorphisms) to
$\Set$, since any linearisation of a poset can be specialised to a
linearisation of any subset of it. Similar constructions are
considered in \cite{Grabowski}. We have
\[\Lin(P\odot Q) \isom \Lin(P)\times\Lin(Q),\qquad
  \Lin(P\sqcup Q) \rightarrow \Lin(P)\times\Lin(Q)\] and
$\Lin(\emptyset)\isom 1$, so the linearisation functor is oplax on the
disjoint union monoidal structure, and strong on the join monoidal
structure. It can be extended to an $n$-fold monoidal functor in the
obvious way.
\end{example}

\subsection{Strict higher categories}
\label{ss:catoids}

Since the beginning of category theory \cites{MacL2, MacL}, it has
been useful to regard a category as a structure on a single set -- its
set of morphisms. One value of this approach is that it can be readily
formalised in non-dependent type theory as used by some proof
assistants~\cite{CMPS}; another one that functors and natural
transformations arise simply as functions.

There are various ways of encoding the relevant structure, and it has
been known for some time that relational monoids offer one convenient
approach~\cites{KenPar}. We begin this section by showing how lax
$n$-fold relational monoids, as defined in Subsections
\ref{ss:lax-monoid-objs} (and Subsection \ref{ss:relational} for
$n=2$) can encode the structure of a strict $n$-category.

Indeed, let $\cC$ be a strict $2$-category. We write $\cC_2$ for the
set of $2$-cells, and this set supports two associative partial
operations. Indeed, for two $2$-cells $\alpha_1:f_1\Rightarrow g_1$
(where $f_1,g_1:x_1\rightarrow y_1$) and $\alpha_2:f_2\Rightarrow g_2$
(where $f_2,g_2:x_2\rightarrow y_2$), we have
\begin{itemize}
\item the horizontal composition $\alpha_1\circ_0\alpha_2$,
  defined if and only if $x_1=y_2$;
\item the vertical composition $\alpha_1\circ_1\alpha_2$,
  defined if and only if $f_1=g_2$ (which implies that $x_1=x_2$ and
  $y_1=y_2$).
\end{itemize}

We consider $\Rel$ as a monoidal bicategory, where an object is a set,
a $1$-cell a relation, a $2$-cell an inclusion of relations, and the
monoidal structure the tensor monoidal structure described in Section
\ref{s:examples-monoidal}.

We claim that the set $\cC_2$ has the structure of a lax double monoid
in $\Rel$ -- as defined in Subsection \ref{ss:relational} --
consisting of two relational monoid structures and relational
inclusions corresponding to the four interchange laws. The double and
triple interchange diagrams are vacuous in this context, since any
diagram of inclusions of relations commutes automatically.

This structure is defined as follows:
\begin{itemize}
\item The first relational monoid structure on the set $\cC_2$,
  written $(\circ_0,\1_0)$, has units consisting of the set of unit
  $2$-cells on all $0$-cells, and multiplication given by horizontal
  composition (where defined);
\item The second relational monoid structure on $\cC_2$, written
  $(\circ_1,\1_1)$, has units consisting of the set of unit $2$-cells
  on all $1$-cells, and multiplication given by vertical composition
  (where defined).
\end{itemize}
This structure satisfies weak interchange. For example, for $\chi$:
given $a,b,c,d$, $(a\circ_1 b)\circ_0(c\circ_1 d)$ is defined if and
only if $s_1a=t_1b$ and $s_1c=t_1d$, and
$s_0a=s_0b=t_0c=t_0d$. Meanwhile, $(a\circ_0 c)\circ_1(b\circ_0 d)$ is
defined if $s_0a=t_0c$, $s_0b=t_0d$, and $(s_1a)(s_1c)=(t_1b)(t_1d)$,
which is a weaker condition.

\begin{remark}
This structure does not make full use of the category $\Rel$, since
the multiplication operations are in fact partial maps of sets. The
units, however, are not partial maps $1\rightarrow\cC_2$ as they are
multivalued, so the whole structure sits in $\Rel$ rather than
$\Set_*$.
\end{remark}

Having used a strict double category to obtain the structure of a lax
double monoid in $\Rel$, we now start going in the other direction: we
reconstruct the structure of a strict double category from a lax
double monoid (in particular, deducing the source and target maps from
the unit morphisms), adapting work of \cites{CMPS, CraDohStr2} to this
categorical context.

Given a lax double monoid $M$ in $\Rel$ where the multiplications are
partial maps, we say that $a\in M$ is a \emph{$0$-cell} if $a\circ_0x
= x$ whenever it is defined and also $x\circ_0 a = x$ whenever it is
defined. We say that it is a \emph{$1$-cell} if $a\circ_1x = x$
whenever defined and $x\circ_1a = x$ whenever defined.

If $a$ and $b$ are $0$-cells, then if $a\circ_0 b$ is defined if it is
equal both to $a$ and $b$, so that $a=b$. For any $x\in M$, there can
be at most one $0$-cell $a$ with $a\circ_0 x$ defined: if $a$ and $b$
both have this property, then
$x = a\circ_0 x = a\circ_0 (b\circ_0 x) = (a\circ_0 b)\circ_0 x$, and
the right-hand side is defined only if $a=b$. If such an $a$ exists,
we say that it is the \emph{$0$-target} of $x$, denoted $t_0(x)$;
similarly if an $0$-cell $a$ exists for which $x\circ_0 a$ is defined,
it is unique and we call it the \emph{$0$-source}, $s_0(x)$.

In fact, in a lax double monoid $M$, $s_0(x)$ and $t_0(x)$ exist for
each element $x\in M$. Indeed, the unit laws for a relational monoid
say for each $x\in M$ that there is some unit $a$ for which
$a\circ_0 x=x$, and that there is some unit $b$ for which
$x\circ_0 b=b$.

Likewise, we can define the \emph{$1$-source} and \emph{$1$-target}
maps $s_1$ and $t_1$. Again, $s_1(x)$ and $t_1(x)$ are well-defined
where they exist, and, again, in the presence of the unit laws of a
lax double monoid they do exist. The functions $s_0$, $t_0$, $s_1$ and
$t_1$ are idempotent; their image is precisely the set of $0$-cells
(for the former two) and $1$-cells (for the latter two).

As strict $n$-categories are formed by pairs of strict $2$-categories
in all different dimensions, we can prove the following fact in a
similar fashion, which says that every strict $n$-category gives a lax
$n$-fold monoid.
\begin{prop}
  The set of $n$-cells of a strict $n$-category can be endowed with
  the structure of a lax $n$-fold monoid in $\Rel$, where $\1_i$
  consists of the units on all $i$-cells, and $\otimes_i$ is just
  $\circ_i$ where defined.
\end{prop}

Again, we can extract notions of $i$-cell, $i$-source and $i$-target
from the resulting structure. A variant of the structure on $\Rel$
just described has been formalised already in the literature under the
name \emph{$n$-catoid} \cites{CMPS}.

\begin{defn}[\cite{CraDohStr2, FahJohStrZie}]
A \emph{catoid} $(C,\odot,s,t)$ consists of a set $C$, a
multioperation $\odot:C\times C\to \bP C$ and source and target maps
$s,t:C\to C$, which satisfy
\begin{gather*}
  \bigcup_{v\in y \odot z} x\odot v = \bigcup_{u\in x \odot y} u \odot
  z,\\ x\odot y \neq \emptyset\Rightarrow s(x)=t(y),\qquad
  x\odot s(x) = \{x\},\qquad t(x)\odot x = \{x\}.
\end{gather*}
\end{defn}

If we extend $\odot$ to powersets as $X\odot Y = \bigcup_{x\in X,y\in
  Y} x\odot y$, we can write the first axiom as $\{x\}\odot (y\odot z) =
(x\odot y)\odot \{z\}$.

Catoids and relational monoids are equivalent; the former presents the
theory in terms of the multioperation and source and target maps; the
latter in terms of the multiplication of ternary relations and units.

We now define the structure that corresponds in a similar fashion to
lax $n$-fold relational monoids.
\begin{defn}
An \emph{$n$-catoid} consists of $n$ catoid structures
$(C,\odot_i,s_i,t_i)_{0\le i <n}$ on the same set $C$, interacting as
follows:
\begin{itemize}
\item for all $i\neq j$,
  \begin{gather*}
    s_i\circ s_j=s_j\circ s_i,\qquad s_i\circ t_j=t_j\circ s_i,\qquad
    t_i\circ s_j=s_j\cdot t_i,\qquad t_i\circ t_j=t_j\circ t_i,\\
    s_i(x\odot_j y)\subseteq s_i(x)\odot_j s_i(y),\qquad t_i(x\odot_j y)\subseteq t_i(x)\odot_j t_i(y),
  \end{gather*}
\item and for all $i<j$,
  \begin{gather*}
    (w \odot_j x)\odot_i (y\odot_j z) \subseteq (w \odot_i y)\odot_j
    (x\odot_i z),\\
    s_j\circ s_i = s_i,\qquad s_j\circ t_i=t_i,\qquad t_j\circ s_i =
    s_i,\qquad t_j\circ t_i = t_i,
  \end{gather*}
\item the \emph{closure condition} that $s_j(s_j(x)\odot_i s_j(y))= s_j(x)\odot_i s_j(y)$ for all $i<j$.
\end{itemize}
\end{defn}

In the presence of the other axioms, the closure condition, which was
absent in a previous definition of $n$-catoid~\cite{CMPS}, is
equivalent to the analogous statement using targets $t_j$ instead of
sources $s_j$.

Since we are working in $\Rel$, the natural notion of morphism of lax
$n$-fold relational monoids is itself a relation that commutes with
the source and target structure.

The following can be shown (and has been checked, on objects, using
the Isabelle/HOL proof assistant).
\begin{thm}
The categories of $n$-catoids and lax $n$-fold
relational monoids are equivalent.
\end{thm}
The closure condition is necessary: there is a three-element
$2$-catoid (produced by Isabelle) without the closure condition which
does not correspond to a lax double relational monoid.

Strict $n$-categories can be represented as a particular class of lax
$n$-fold relational monoids. It is therefore natural to ask which
$n$-catoids correspond to strict $n$-categories. Two definitions
single out those which do.
\begin{defn}[\cite{CraDohStr2, FahJohStrZie}]
A catoid $C$ is \emph{functional} if $x,x'\in y \odot z\Rightarrow x=x'$ for
all $x,x',y,z\in C$. It is \emph{local} if $t(x)=s(y)\Rightarrow
x\odot y \neq \emptyset$, and therefore $x\odot y\neq \emptyset
\Leftrightarrow t(x)=s(y)$.
\end{defn}
Obviously, locality imposes the composition pattern of arrows in
categories while functionality makes the multioperation in catoids a
partial operation.  It is already evident from the work of Mac Lane
\cite{MacL} on single-set categories that categories are equivalent to
local functional catoids.

\begin{defn}
  An $n$-catoid is \emph{functional} if each of the $n$ underlying
  catoids is. It is \emph{local} if each of the underlying catoids is.
\end{defn}

Local functional $n$-catoids satisfy the closure condition, and we
have checked using Isabelle that strict $2$-categories are equivalent
to local functional $2$-catoids. Since in both cases all the axioms
mention only two dimensions at a time, we have the following result.

\begin{thm}
Strict $n$-categories are equivalent to local functional $n$-catoids.
\end{thm}

\section{Algebra over the Boolean semiring}
\label{s:boolsemiring}

In this section, we explore a rich source of examples of the theory
above: algebraic, coalgebraic and bialgebraic structures in
(semi)lattices. Since we have taken the time to provide good
categorical machinery, it is possible to do much of our work uniformly
with only a few raw materials: we can regard semilattices as modules
over the Boolean semiring.

We start by laying out the relevant algebraic material. We do this in
as similar a fashion as possible to the better-known theory of
rings. Connes and Consani \cite{ConCon} provide similar constructions;
see also \cite{ImKhov} for an application of similar concepts in
automata theory. Much of this material provides examples of more
general categorical constructions, but there is enough generality
elsewhere and we have no need of it here.

\subsection{Semilattices and the boolean semiring}

Let $\bB$ denote the semiring of Booleans, the quotient of the free
semiring $\bN$ on no generators by the relation $1+1=1$.  A
\emph{$\bB$-module} is an abelian monoid $(M,+,0)$ with an action
$\bB\times M\to M$, satisfying the usual laws.

It is the same thing as an idempotent commutative monoid and hence as
a $\vee$-semilattice with join $\vee=+$ and least element
$0$. Idempotency follows from $1x+1x = (1+1)x = 1x$. Defining
$x\leq y$ if and only if $y=x+y$ gives the standard semilattice order.

The algebra of $\bB$ is as fundamental to the study of semilattices as
the algebras of $\bN$ and $\bZ$ are to that of their modules:
commutative monoids and abelian groups, respectively.

We write $\Mod_\bB$ for the category of $\bB$-modules (with the
evident notion of homomorphism). Each finite $\bB$-module is a
fortiori a complete semilattice and hence a complete lattice with
$x\wedge y = \sum_{z\leq x, z\leq y}z$. In much of what follows we
focus on finite modules.

The forgetful functor from the category of complete semilattices (and
join-preserving maps) to the category of posets has a left adjoint
$\bD$. For every finite poset $P$, $\bD P$ is the set of down-closed
subsets of $P$, where join is union, meet is intersection, the bottom
element is the empty set and the top element is $P$. One can think of
this construction as freely adding joins: a down-closed subset $X$
represents a join $\bigvee_{x\in X}x$ and, from this point of view,
the reason why we restrict to down-closed subsets is that a join over
a set is unchanged by down-closure. If $P$ is a finite poset, then all
joins are finite, and hence $\bD P$ a finite complete semilattice.

There is also an adjoint $\fin{\bD}$ to the forgetful functor from
semilattices (and maps which preserve finite joins) to posets. In this
case we take only finite joins, and so $\fin{\bD} P$ can be thought of
as the finitely-generated down-closed subsets. Of course, for finite
posets, $\bD$ and $\fin{\bD}$ agree.

The complement of every down-closed subset is evidently
up-closed. Further, complementation reverses inclusion and exchanges
unions and intersections.

\begin{lemma}
  For any poset $P$, $(\bD(P^{\op}))^\op\isom\bD P$.  \qed
\end{lemma}

There is no analogue of this for $\fin{\bD}$: complements of
finitely-generated down-closed subsets can be infinitely-generated
up-closed subsets.

The category of $\bB$-modules has cartesian products; in keeping with
traditional notation for linear algebra we write $\oplus$ for the
product --  the \emph{direct sum}.

By contrast, the \emph{tensor product} of two $\bB$-modules $M$ and
$N$ is the $\bB$-module $M\otimes N$ generated by pairs $(m, n)$ for
$m\in M$ and $n\in N$, subject to the relations that
$(m, 0) = 0 = (0, n)$, $(m_1 + m_2, n) = (m_1, n) + (m_2, n)$ and
$(m, n_1 + n_2) = (m, n_1) + (m, n_2)$. The order on $M\otimes N$ is
given as $\sum_i (m_i, n_i)\leq \sum_j (m_j', n_j')$ if and only if
for each $i$ there is a $j$ such that $m_i\leq m_j'$ and
$n_i\leq n_j'$. This construction is identical to the better-known
construction of tensor products of modules over a commutative ring or
vector spaces over a field. In the language of semilattices, it has
been considered previously \cites{AndKim, GraLakQua}.

The tensor product yields a monoidal structure on $\Mod_\bB$. It is
defined on homomorphisms $f:M\rightarrow M'$ and $g:N\rightarrow N'$
by extending the formula
\[(f\otimes g)(m,n) = (f(m),g(n))\]
linearly. The unit is $\bB$ itself, thanks to the isomorphism
$M\rightarrow B\otimes M$ sending $m\mapsto (1,m)$. It thus differs
from the cartesian monoidal structure.

\begin{example}
If, for example, the $\bB$-modules $M$ and $N$ are both three-element
chains as depicted below, then their tensor product has six elements
(where we have shortened our pair notation $(m, n)$ to $mn$ for
brevity):
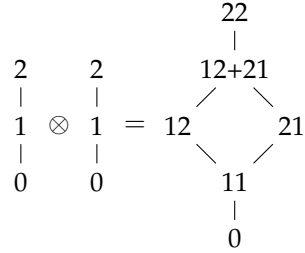
\begin{displaymath}
  \begin{tikzpicture}
    \begin{scope}[scale=0.75, shift={(-2,0)}]
    \node (0) at (0,-1) {0};
    \node (1) at (0,0) {1};
    \node (2) at (0,1) {2};
    \draw (0) -- (1);
    \draw (1) -- (2);
    \end{scope}
    \node at (-1,0) {$\otimes$};
    \begin{scope}[shift={(-0.5,0)}, scale=0.75]
    \node (0) at (0,-1) {0};
    \node (1) at (0,0) {1};
    \node (2) at (0,1) {2};
    \draw (0) -- (1);
    \draw (1) -- (2);
    \end{scope}
    \node at (0,0) {$=$};
    \begin{scope}[scale=0.75, shift={(1.75,0)}]
    \node (0) at (0,-2) {0};
    \node (11) at (0,-1) {11};
    \node (12) at (-1,0) {12};
    \node (21) at (1,0) {21};
    \node (1221) at (0,1) {12+21};
    \node (22) at (0,2) {22};
    \draw (0) -- (11);
    \draw (11) -- (12);
    \draw (11) -- (21);
    \draw (12) -- (1221);
    \draw (21) -- (1221);
    \draw (1221) -- (22);
    \end{scope}
  \end{tikzpicture}
\end{displaymath}
\end{example}

In fact, more is true.
\begin{prop}
The category of $\bB$-modules is monoidal closed, with internal hom
$\Hom(A,B)$ given by the set of $\bB$-module homomorphisms from $A$ to
$B$, with $\bB$-module operations defined pointwise.\qed
\end{prop}

This means in particular that there is an isomorphism
\[\Hom(\Hom(A\otimes B),C)\isom\Hom(A,\Hom(B,C)).\]

\subsection{Multiplicative structure on $\bB$-modules}

A \emph{$\bB$-algebra}, a monoid object in $\bB$-modules, is nothing
but a \emph{dioid}: an additively idempotent semiring, or equivalently
a $\vee$-semilattice with a multiplication distributing over joins in
both arguments. We write $\Alg_\bB$ for the category of
$\bB$-algebras. If the underlying $\bB$-module is a complete
semilattice, then the algebra is known as a (unital) \emph{quantale}
\cite{RosenthalQ}.

The tensor product $A\otimes B$ of $\bB$-algebras $A$ and $B$ inherits
a $\bB$-algebra structure with product defined by
$(a_1,b_1)(a_2,b_2) = (a_1a_2,b_1b_2)$ on generators; it extended to
$A\otimes B$ by bilinearity. This yields a monoidal structure on
$\Alg_\bB$.

For a $\bB$-algebra $A$, we can define left and right $A$-modules in a
similar fashion. If $A$ is commutative these concepts coincide, and we
can define an $A$-algebra as a monoid object in $A$-modules. Such
modules and algebras have an underlying $\bB$-module or $\bB$-algebra
structure, respectively. In this case there is a tensor product
$M\otimes_A N$ for $A$-modules, which also has the relation that
$(ma)n = m(an)$ for $a\in A$, $m\in M$ and $n\in N$. Here the products
$ma$ and $an$ are from the structures of $A$-module on $M$ and $N$,
while the other products are the products defining $M\otimes_A N$.

For any $\bB$-algebra $A$, there is a forgetful functor from
$A$-modules to sets. It has a left adjoint, the \emph{free $A$-module}
on generating set $X$, written $A[X]$. This object, which in fact has
a natural $A$-algebra structure, is the set whose elements are maps of
sets $X\rightarrow A$ with \emph{finite support} (they are zero on all
but a finite subset of $X$), with sum and module action defined
pointwise. When $A=\bB$, we can view $\bB[X]$ as the set of finite
subsets of $X$, with sum as union and product as intersection. If $X$
is empty, then $A[X]$ is the terminal lattice (with one element); if
$X$ is a singleton, then $A[X]\isom A$.

We also use a related construction: we write $A^X$ for the collection
of all maps $X\rightarrow A$ with no restriction on support. This is
again an $A$-algebra with the same sum and product operations. We can
similarly regard $\bB^X$ to be the set of all subsets of $X$, with sum
as union and product as intersection. One key use of this construction
is that $\Hom(\bB[X], A)\isom A^X$: both sides are defined by maps of
sets from $X$ to $A$. Since all maps from finite sets have finite
support, we also have $A[X]\isom A^X$ whenever $X$ is finite.

These isomorphisms enable us to prove some structural properties of
the free module construction.
\begin{prop}
  The free module construction has two strong monoidal structures, one
  for each of the two monoidal structures on $\bB$-modules (the
  cartesian product $\oplus$, and the tensor product $\otimes$). In
  particular, we have $\bB[\emptyset]\isom 1$, $\bB[1]\isom\bB$, and
  \begin{align}
    \bB[X\sqcup Y] &\isom \bB[X]\oplus\bB[Y],\label{eq:Bmod-freesum}\\
    \bB[X\times Y] &\isom \bB[X]\otimes\bB[Y].\label{eq:Bmod-freeprod}
  \end{align}
\end{prop}
\begin{proof}
  We study both identities in terms of the maps out of these, for any
  sets $X$ and $Y$, into any $\bB$-module $A$. For
  (\ref{eq:Bmod-freesum}),
\begin{align*}
  \Hom(\bB[X\sqcup Y],A)
  &\isom A^{X\sqcup Y}\\
  &\isom A^X\times A^Y\\
  &\isom \Hom(\bB[X],A)\times\Hom(\bB[Y],A)\\
  &\isom\Hom(\bB[X]\oplus\bB[Y],A).
\end{align*}
For  (\ref{eq:Bmod-freeprod}),
\begin{align*}
  \Hom(\bB[X\times Y], A)
  &\isom A^{X\times Y}\\
  &\isom {(A^Y)}^X\\
  &\isom \Hom(\bB[X],A^Y)\\
  &\isom \Hom(\bB[X],\Hom(\bB[Y],A))\\
  &\isom\Hom(\bB[X]\otimes\bB[Y], A).
\end{align*}
The required isomorphisms then follow from the Yoneda lemma.
\end{proof}

There is also the following relation between free modules and tensors.
\begin{prop}
  Let $A$ be a $\bB$-module. Then $A[X]\isom A\otimes\bB[X]$.
\end{prop}
\begin{proof}
  We consider again homomorphisms into a $\bB$-module $B$. Thus
  \begin{align*}
    \Hom(A[X], B)
    &\isom \Hom(A, B^X)\\
    &\isom \Hom(A, \Hom(\bB, B^X))\\
    &\isom \Hom(A, \Hom(\bB[X], B)) \\
    &\isom \Hom(A\otimes\bB[X], B),
  \end{align*}
  and the isomorphism follows again from Yoneda.
\end{proof}

We need the following further properties of the $\Hom$ construction.
\begin{prop}
  \label{prop:bmod-hom-lax}
Using the tensor product monoidal structure, the functor
$\Hom:\Mod_\bB^\op\times\Mod_\bB\rightarrow\Mod_\bB$ is lax monoidal.
\end{prop}
\begin{proof}
The unit is a morphism $\bB\rightarrow\Hom(\bB,\bB)$; we can take
this to be the identity map. The multiplication is a map
\[\Hom(M_1,N_1)\otimes\Hom(M_2,N_2)\longrightarrow\Hom(M_1\otimes
M_2,N_1\otimes N_2).\]
On objects it is defined by sending $f_1\otimes f_2$ to the map
sending $m_1\otimes m_2$ to $f_1m_1\otimes f_2m_2$; it is quick to
check that this is well-defined.
\end{proof}

We are able to identify the algebraic dual $\Hom(M,\bB)$ with the
better-known combinatorial dual $M^\op$.
\begin{prop}
  Let $M$ be a finite $\bB$-module. Then the $\bB$-module
  $\Hom(M,\bB)$ is the dual lattice $M^\op$.
\end{prop}

\begin{proof}
  Any homomorphism $f:M\rightarrow\bB$ is determined by the set
  $f^{-1}(0)$, which is join-closed. Every element of $f^{-1}(0)$ is
  below $s_f=\sum_{x\in f^{-1}(0)}x$ by definition, and
  $s_f\in f^{-1}(0)$ as well, because
  $f(s_f) = \sum_{x\in f^{-1}(0)}f(x)=\sum_{x\in f^{-1}(0)}0$.  Hence
  $f$ is determined by the property that the elements it sends to zero
  is exactly the down-closure of $s_f$, that is,
  $f^{-1}(0) = s_f{\downarrow}$.

Conversely, for any $x\in M$, define $f_x:M\rightarrow\bB$ by
\begin{equation*}
  f_x(y) =
  \begin{cases}
    0 &\text{ if } y\le x,\\
1 & \text{ otherwise}.
  \end{cases}
\end{equation*}
This establishes a bijection between the homomorphisms $M\rightarrow
\bB$ and the elements of $M$.

Moreover, $f_x\leq f_y$ if and only if $y\leq x$, so functions form
the dual lattice. Also, pointwise addition corresponds to join in the
dual lattice (which is meet in the original lattice), since
$f_x(z)+f_y(z)=0$ if and only if $z\leq x\wedge y$, which is the case
if and only if $f_{x\wedge y}(z)=0$.
\end{proof}

\subsection{Meets}

It is worth asking to what extent our constructions preserve lattice
structure, in particular completeness.

This is not always the case. Indeed, even though $\bB$ is a lattice,
$\bB[X]$ is not a lattice if $X$ is an infinite set, as the top
element would have to be the constant top function $X\rightarrow\bB$,
and this is not equal to the bottom element on all but a finite subset
of $X$.

However, we have the following two facts.
\begin{prop}
  If $m_1$ and $m_2$ have a meet $m_1\wedge m_2$ in $M$, and $n_1$ and
  $n_2$ have a meet $n_1\wedge n_2$ in $N$, then $(m_1\wedge
  m_2)(n_1\wedge n_2)$ is a meet of $m_1n_1$ and $m_2n_2$ in $M\otimes
  N$.
\end{prop}
\begin{proof}
  Consider an element $\sum_i m'_in'_i$ in $M\otimes N$. It is less
  than $m_1n_1$ if and only if $m'_i\leq m_1$ and $n'_i\leq n_1$ for
  all $i$; similarly it is less than $m_2n_2$ if and only if $m'_i\leq
  m_2$ and $n'_i\leq n_2$ for all $i$.

  Both are the case if and only if $m'_i\leq m_1\wedge m_2$ and
  $n'_i\leq n_1\wedge n_2$ for all $i$, which in turn happens if and
  only if $\sum_im'_in'_i\leq (m_1\wedge m_2)(n_1\wedge n_2)$. This is
  exactly the defining property of a meet.
\end{proof}

\begin{prop}
  If $X$ and $Y$ are lattices, then so is $X\otimes Y$.
\end{prop}
\begin{proof}
  We must provide a meet for arbitrary elements $\sum_im_in_i$ and
  $\sum_jm'_jn'_j$. Yet distributivity of meets and joins tells us
  that $\sum_{i,j}(m_i\wedge m'_j)(n_i\wedge n'_j)$ works.
\end{proof}

\subsection{Algebraic structures on free $\bB$-modules}
\label{ss:asb}

Our interest in $\bB$-modules is partly that free $\bB$-modules yield
an algebraic description of binary relations. By the results of the
previous section, relations from $X$ to $Y$ are functions
$X\times Y\rightarrow\bB$, hence $\bB$-module morphisms from
$\bB[X\times Y]$ to $\bB$, $\bB$-module morphisms from
$\bB[X]\otimes\bB[Y]$ to $\bB$ and finally $\bB$-module morphisms from
$\bB[X]$ to $\bB^Y$.

The category $\Rel$ can in general also be regarded as the subcategory
of $\Mod_\bB$ whose objects are the $\bB$-modules $\bB^X$ and whose
morphisms preserve all joins (not just the finite ones). This has been
described in the proof of Theorem 1.1 in \cite{Rosenthal}, and the
beginning of \cite{KenPar}.

However, in the presence of finiteness conditions, this algebraic
approach works out more pleasantly. The category of finite sets and
relations can be regarded as the subcategory of $\Mod_\bB$ on the
finitely generated free $\bB$-modules. Thus we assume $X$ to be finite
in what follows.

A dioid structure whose underlying $\bB$-module is a free $\bB$-module
$\bB[X]$ has a product $\bB[X\times X]\rightarrow\bB[X]$, which is the
same thing as a ternary relation on $X$, and a unit
$\bB\rightarrow\bB[X]$, which is the same thing as a subset of $X$;
this is thus a relational monoid as considered in Subsection
\ref{s:examples-monoidal}.

Moreover, the free $\bB$-module construction can be extended to a
\emph{lax monoidal functor} \cite{JohYau} from the monoidal bicategory
$\Rel$ to the monoidal bicategory $\Mod_\bB$. This preserves $n$-fold
lax monoid objects (as considered in subsection \ref{ss:catoids}): a
relational $n$-fold monoid is taken to a lax $n$-fold monoid in
semilattices.

\begin{example}
\label{ex:meets-as-product}
One source of examples that sets this situation apart from algebras
over rings or other semirings is that some finite $\bB$-module have
natural $\bB$-algebra structure: (as has been mentioned) every finite
semilattice is a lattice. If that lattice is distributive, then meets
provide a multiplication.
\end{example}

\begin{example}
\label{ex:shuffle-languages}
As one example of particular interest, the lax double monoid structure
given by concatenation and shuffle product of words, which was
considered earlier in Subsection \ref{ss:relational}, can be lifted to
the set $\bP M$ of languages, which is the free $\bB$-module on the
set of words.

The first monoid structure lifts concatenation to the complex product
\[X\cdot Y = \{c\mid a\in X, b\in Y,c = ab\},\]
and the second lifts shuffle to the complex product
\[X\shuffle Y = \{c\mid a\in A, b\in B, c\in a\shuffle b\}.\]
Both products have the same unit $1=\{e\}$, the empty word language.

The lax interchange law on languages
\begin{align*}
  (W\cdot X)\shuffle(Y\cdot Z) \subset (W\shuffle Y)\cdot(X\shuffle Z)
\end{align*}
is then the interchange law of the lax double monoid structure, which
is in fact a concurrent monoid in the sense of~\cite{HoaMolStrWeh}.
\end{example}

\subsection{Coalgebras}
\label{ss:coalgebras}

We have considered $\bB$-algebras as a type of monoid; we turn now to
the corresponding concept of comonoid.

We define a \emph{$\bB$-coalgebra} (which might reasonably be called a
\emph{codioid}) to be a comonoid object in the category of
$\bB$-modules (with tensor product). This is analogous to how
$\bB$-algebras and hence dioids are monoid objects in the category of
$\bB$-modules.

$\bB$-coalgebras with underlying $\bB$-module of the form $\bB^X$
(with $X$ finite) are equivalent to $\bB$-algebras with the same
underlying $\bB$-module. Indeed, such a thing consists of a relation
$X\rightarrow X\times X$ and another relation $X\rightarrow 1$,
satisfying relations opposite to those of a dioid. Since the category
of relations is self-opposite, they are the same.

However, for other kinds of $\bB$-modules, they may be different
structures. For example, the five-element $\bB$-module $M$ defined as
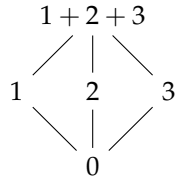
\begin{displaymath}
  \begin{tikzpicture}
    \node (0) at (0,-1) {$0$};
    \node (1) at (-1,0) {$1$};
    \node (2) at (0,0) {$2$};
    \node (3) at (1,0) {$3$};
    \node (123) at (0,1) {$1+2+3$};
    \draw (0) -- (1);
    \draw (0) -- (2);
    \draw (0) -- (3);
    \draw (1) -- (123);
    \draw (2) -- (123);
    \draw (3) -- (123);
  \end{tikzpicture}
\end{displaymath}
can be thought of as the collection of subsets of $\{1,2,3\}$ that
contain all three if they contain any two. Similarly, the tensor
product $M\otimes M$ can be thought of as the collection of subsets of
$\{1,2,3\}^2$ with the property that, if they contain any two points
on a horizontal or vertical line, then they contain all three. It can
quickly be checked that there are fifty such subsets.

$M$ has fifty commutative $\bB$-algebra structures: probably the most
obvious of them is the one defined by intersection of
subsets. However, it does not carry any cocommutative $\bB$-coalgebra
structures.

\begin{example}
A monoid $M$ is \emph{finitely decomposable} if, for any $x\in M$,
there are only finitely many pairs $(y_1,y_2)$ for which
$y_1y_2=x$. Examples include all finite monoids and all free
monoids. An example of a monoid which is not finitely decomposable is
given by the integers $\bZ$ under addition: $x+(-x)=0$ for any $x$.

Any finitely decomposable monoid $M$ yields a $\bB$-coalgebra
structure on $\bB[M]$, with comultiplication defined by
\[m(x)=\sum_{\substack{y_1,y_2\in M\\y_1y_2=x}}y_1y_2.\]
It is cocommutative if $M$ is commutative.

The preceding construction works over any semiring (including the
integers), but there is a variant which works only over the special
case of the booleans. If $M$ is an ordered finitely decomposable
monoid, we can instead obtain a $\bB$-coalgebra structure on $\bB[M]$
with comultiplication defined by
\[m'(x)=\sum_{\substack{y_1,y_2\in M\\y_1y_2\leq x}}y_1y_2.\]
Again, it is cocommutative if $M$ is commutative.

The first example is coassociative over any semiring because, for any
$x$, there is a bijection between solutions to the equations $aw=x$,
$uv=a$ and solutions to the equations $ub=x$, $vw=b$ (both are just a
trio of elements $u$, $v$, $w$ with $uvw=x$). However, the second
example is not coassociative over $\bZ$. Indeed, there is in general
no bijection between solutions to the above if the equations are
replaced by inequalities. However, there is a bijection between
triples $u$, $v$, $w$ that work (we just seek $uvw\leq x$) ensuring
coassociativity over $\bB$.
\end{example}

\begin{example}
The incidence coalgebras from \cite{JonRot} can be adapted to the
setting of $\bB$-modules. The \emph{incidence coalgebra} $I_P$ on a
poset $P$ is the free $\bB$-module generated by expressions $[x,y]$
for all $x\leq y\in P$. The coproduct and counit are given by
\begin{displaymath}
  \Delta[x,y] = \sum_{x\leq z\leq y}[x,z]\otimes[z,y],\qquad
  \varepsilon[x,y] =
  \begin{cases}
    1, & \text{if $x=y$};\\
    0, & \text{otherwise}.
  \end{cases}
\end{displaymath}
Incidence coalgebra are usually defined over fields. But they make no
reference to the subtraction, and hence are defined over semirings
such as $\bB$ here.
\end{example}

\begin{example}
  \label{ex:incidence}
  Variants to the incident coalgebra can be made more specific to
  $\bB$. Let $P$ be a finite poset. We write $C_P$ for the set of
  convex subsets of $P$: the subsets such that, if $x\leq z\leq y$ and
  $x,y\in P$, then $z\in P$ also.

  The set $C_p$ forms a semilattice under intersections and thus a
  $\bB$-module. It can be made into a $\bB$-coalgebra in two different
  ways, $C_P^+$ and $C_P^-$:
\begin{itemize}
\item the \emph{strictly ordered coproduct}, where $\Delta^+(A)$ is
  the sum over $B\otimes C$, where $A$ is the disjoint union of $B$
  and $C$ and $b\leq c$ for every $b\in B$ and $c\in C$.
\item the \emph{weakly ordered coproduct}, where $\Delta^-(A)$ is the
  sum over $B\otimes C$, where $A$ is the disjoint union of $B$ and
  $C$ and there do not exist $b\in B$ and $c\in C$ with $b>c$.
\end{itemize}
In both cases the counit is given by $\varepsilon(A)=1$ if $A$ is
empty and $0$ otherwise. Also, in both cases, the compatibility of the
coproduct with the $\bB$-module structure can be checked quickly.
\end{example}

\subsection{Double algebras}
\label{ss:double}

We define a \emph{double $\bB$-algebra} (in the literature on
concurrency theory, such as in \cite{HoaMolStrWeh}, known as a
\emph{concurrent dioid}, or even as a \emph{trioid}) to be a lax
double monoid object in $\bB$-modules.

In more detail, it is a $\bB$-module with two products $\cdot_1$ and
$\cdot_2$, with products $1_1$ and $1_2$ respectively, satisfying
interchange laws which state:
\begin{itemize}
\item $(w \cdot_2 x)\cdot_1 (y\cdot_2 z) \leq (w \cdot_1 y)\cdot_2(x\cdot_1 z)$;
\item $1_2\cdot_11_2\leq 1_2$;
\item $1_1\leq 1_1\cdot_21_1$;
\item $1_1\leq 1_2$.
\end{itemize}

\begin{example}
Given a relational double monoid as considered in Subsection
\ref{ss:relational} above, we can apply the free $\bB$-module
construction, discussed in Subsection \ref{ss:asb} to get a double
$\bB$-algebra.

For example, given the relational double monoid on $M=A^*$ given by
concatenation and shuffle of words, we get the double $\bB$-algebra of
languages, with operations consisting of the liftings of concatenation
and shuffle to the language level. This systematises the construction
in Example \ref{ex:shuffle-languages}.
\end{example}

\begin{example}
Another example is given by meets in a monoidal lattice: a lattice $M$
with a monoidal structure $\cdot,1$ such that
\begin{align*}
  (a\cdot b)\vee(c\cdot d)&\leq (a\vee c)\cdot(b\vee d),\\
  1\vee 1&\leq 1.
\end{align*}
It follows from the first equation that $\cdot$ is increasing in both variables: if $a\leq a'$ and $b\leq b'$, then
\[a\cdot b\leq (a\cdot b)\vee(a'\cdot b')\leq (a\vee a')\cdot(b\vee b') = a'\cdot b'.\]

This structure is, by construction, a $\bB$-algebra, and the meet
automatically extends it to a double $\bB$-algebra: the identity
\[(a\cdot b)\wedge(c\cdot d)\geq (a\wedge c)\cdot(b\wedge d)\]
is satisfied since both terms of the left-hand side are greater than
the right-hand side by increasingness of $\cdot$.
\end{example}

\begin{remark}
It is also possible to define a $\bB$-bialgebra as a $\bB$-module with
a $\bB$-algebra structure and a $\bB$-coalgebra structure and a lax
interaction between them. Examples can be made, for example, by taking
the incidence coalgebra (as defined in Example \ref{ex:incidence}) of
a monoidal poset and obtaining a $\bB$-algebra structure induced by
the monoidal structure, but we will not pursue this here.
\end{remark}

\subsection{Convolution}

\begin{defn}
  For every monoid object $M$ and comonoid object $C$ in a monoidal
  category $\cC$ there is a \emph{convolution product} on morphisms
  $C\rightarrow M$, making them into a monoid. The unit is the
  morphism $C\stackrel{u}{\rightarrow}1\stackrel{u}{\rightarrow}M$
  obtained by composing the counit and unit. The product of
  $f,g:C\rightarrow M$ is the composite
\begin{center}
  \begin{tikzcd}
    C\ar[r, "m"]&C\otimes C\ar[r, "f\otimes g"]&M\otimes M\ar[r, "m"]&M
  \end{tikzcd}
\end{center}
obtained by composing the comultiplication, the tensor of the two
maps, and the multiplication map.
\end{defn}
It is easy to check that the convolution product is associative and
unital. Also, if $\cC$ is symmetric monoidal, and $C$ is a
cocommutative comonoid object, and $M$ is a commutative monoid object,
then the convolution product is commutative.

\begin{example}
  As an example, working in the category of $\bB$-modules we can take
  the \emph{dual} $\Hom_\bB(C,\bB)$ of a $\bB$-coalgebra $C$.

  If $C$ is isomorphic to $\bB^X$ for a finite set $X$, then
  $\Hom_\bB(C,\bB)$ is also isomorphic to $\bB^X$. However, for larger
  $\bB$-coalgebras, we get something much bigger.

  Using this construction, we can reconstruct much of the theory of
  languages. Indeed, if $A$ is a set and $M$ the free monoid on $A$
  (thought of as the set of words on alphabet $A$), then since $M$ is
  finitely decomposable, $\bB^M$ has a $\bB$-coalgebra structure given
  by
  \[w\mapsto\sum_{w_1w_2=w}(w_1,w_2).\]
  This is non-cocommutative (unless $A$ is empty or a singleton).

  The dual $\bB$-algebra $\Hom_\bB(\bB^M,\bB)$ can be thought of as
  the collection of formal noncommutative power series with
  coefficients in $\bB$: since a map of $\bB$-modules
  $\bB^M\rightarrow\bB$ is determined by the map of sets
  $M\rightarrow\bB$, we get a coefficient for each word, and the
  multiplication rule is the usual noncommutative power series
  multiplication.  It can equally be thought of as the set of
  languages over the alphabet $A$: for each word in $M$, a
  homomorphism $f:\bB^M\rightarrow\bB$ classifies whether that word is
  in the language or not.
\end{example}

This situation generalises considerably. Let $\cC$, $\cD$ and $\cE$ be
monoidal categories. The category $\cC^\op\times\cD$ has a natural
monoidal structure. A lax functor from $1$ then consists of a comonoid
object in $\cC$ and a monoid object in $\cD$. Its composite with any
lax monoidal functor $H:\cC^\op\times\cD\rightarrow\cE$ gives a monoid
object in $\cE$.

Above we considered the special case of $\cD=\cC$, $\cE=\Set$ (with
the cartesian monoidal structure), and $H(X,Y)$ being the set of
morphisms $\cC(X,Y)$ from $X$ to $Y$. This functor $H$ has a lax
structure, where $\varepsilon:1\rightarrow\cC(\I,\I)$ picks out the
identity function $\I\rightarrow\I$, and
$\mu:\cC(X_1,Y_1)\times\cC(X_2,Y_2)\rightarrow\cC(X_1\otimes
X_1,Y_1\otimes Y_2)$ sends $(f,g)$ to $f\otimes g$.

Another simple example is given by the lax monoidal functor from
Proposition \ref{prop:bmod-hom-lax}. This tells us that if $M$ is a
$\bB$-algebra, and $C$ is a $\bB$-coalgebra, then the set of
$\bB$-module maps $\Hom_\bB(C,M)$ has the natural structure of a
$\bB$-algebra.


We can generalise further, using fully lax functors of $n$-fold
monoidal categories. If $\cC$, $\cD$ and $\cE$ are in fact symmetric
monoidal, then any lax monoidal functor
$\cC^\op\times\cD\rightarrow\cE$ is a fully lax functor of $n$-fold
monoidal categories. This gives in particular that the convolution of
an $n$-fold comonoid object with an $n$-fold monoid object in a
symmetric monoidal category is an $n$-fold monoid object.

\subsection{Graded objects}

We can produce a kind of categorification of the example
$\Hom(\bB^M,\bB)$. Let $M$ be a monoid and $\cC$ a category with
coproducts (if $M$ is finitely decomposable, it suffices to require
that $\cC$ have finite coproducts). Then the category of
\emph{$M$-graded objects of $\cC$} is the product category $\cC^M$. It
also has coproducts, which are computed pointwise.

If $\cC$ is a symmetric monoidal category such that $\otimes$
distributes over coproducts (and, again, if $M$ is finitely decomposable, we
need only that $\otimes$ distributes over finite coproducts) then
$\cC^M$ becomes a monoidal category with the \emph{graded unit}
defined by $\I_1 = \I$ and $\I_m = 0$ otherwise (where $0$ is an
initial object) and \emph{graded product} defined by
\[(X\otimes Y)_m = \bigsqcup_{n_1n_2=m}\left(X_{n_1}\otimes Y_{n_2}\right).\]
It also has the \emph{pointwise monoidal structure}, as it is simply a
product of monoidal categories, defined by $\1_m=\1$ and $(X\odot Y)_m
= X_m\otimes Y_m$.

We claim that these form a double monoidal structure. Indeed, the map
\[\chi:(W\odot X)\otimes(Y\odot Z)\longrightarrow(W\otimes Y)\odot(X\otimes Z)\]
has for its $m$th component a map
\[\bigsqcup_{n_1n_2=m}W_{n_1}\otimes X_{n_1}\otimes Y_{n_2}\otimes Z_{n_2}\longrightarrow\left(\bigsqcup_{n_1n_2=m}W_{n_1}\otimes Y_{n_2}\right)\otimes\left(\bigsqcup_{n_1n_2=m}X_{n_1}\otimes Z_{n_2}\right),\]
which, upon multiplying out using distributivity, is just an inclusion
of components.

The map $\zeta:\I\rightarrow\I\odot\I$ is a natural isomorphism, while
the map $\nu:\1\otimes\1\rightarrow\1$ has $m$th component given by
the fold map
\[\sqcup_{n_1n_2=m}\I\rightarrow\I,\]
and lastly the map $\iota:\I\rightarrow\1$ is the identity (on the
$1$st component) and the initial map (on every other component).

In fact, in the usual way of the examples above, this can be extended
to an $n$-fold monoidal structure, by using several copies of each of
these two monoidal structures described here.

In homological algebra \cite{Weibel}, heavy use is made of graded
abelian groups; in such cases $M$ is usually the additive monoid of
the natural numbers, the integers, or the integers modulo two.

\countgrumbles
\end{document}